\documentclass[11pt, a4paper]{article}
\usepackage{graphicx} 
\usepackage{hyperref} 

\usepackage{authblk} 

\usepackage{graphicx}
\usepackage{amsmath,amssymb,amsthm}
\usepackage{mathrsfs}
\usepackage{mathabx}\changenotsign
\usepackage{dsfont}
\usepackage{xcolor}
\usepackage{enumitem}

\usepackage[T1]{fontenc}
\usepackage{lmodern}
\usepackage[babel]{microtype}
\usepackage[english]{babel}

\usepackage{geometry}
\numberwithin{equation}{section}
\numberwithin{figure}{section}

\newtheorem{theorem}{Theorem}
\newtheorem{conjecture}[theorem]{Conjecture}

\newtheorem{lemma}[theorem]{Lemma}

\newtheorem{problem}[theorem]{Problem}

\newtheorem{remark}[theorem]{Remark}
\newtheorem{claim}[theorem]{Claim}

\newcommand{\eps}{\varepsilon}

\newcommand{\Prob}{\operatorname{Pr}}
\newcommand{\cF}{\mathcal{F}}
\newcommand{\cG}{\mathcal{G}}

\title{Maker playing against an invisible Breaker}

\date{}

\author[1]{Dennis Clemens}
\author[1]{Fabian Hamann}
\author[2]{Mirjana Mikala\v{c}ki}
\author[3]{Yannick Mogge}
\author[2]{Milo\v{s} Stojakovi\'c}

\affil[1]{Hamburg University of Technology, Institute of Mathematics, Am Schwarzenberg-Campus 3, 21073 Hamburg, Germany}
\affil[2]{Department of Mathematics and Informatics, Faculty of Sciences, University of Novi Sad, Serbia}
\affil[3]{Universite Claude Bernard Lyon 1, CNRS, INSA Lyon, LIRIS, UMR5205, 69622 Villeurbanne, France}

\begin{document}

\maketitle


\begin{abstract}
We initiate the study of the phantom version of Maker-Breaker positional games. In a phantom game, the moves of one of the players are hidden from the other player, who still has the complete information. We look at the biased $(a:b)$ Maker-PhantomBreaker games where the board is the edge set of the complete graph on $n$ vertices, $K_n$, and Maker has no information about PhantomBreaker's choices of edges. We give randomized strategies for both players in four classical games: connectivity game, perfect matching game, mindegree-$k$ game and Hamiltonicity game. In particular, we focus on characterizing all biases $(a:b)$ for which Maker wins asymptotically almost surely.
\end{abstract}

\section{Introduction}
The study of combinatorial games with imperfect information leads to a very challenging field. One family of such games are those in which the actions of one player are invisible for the opponent, like \textit{Phantom Tic-Tac-Toe} (see, e.g.,~\cite{teytaud2011lemmas,wang2015belief}) and \textit{Phantom Go}, which has become a classical challenge in AI and it is a part of the annual computer Olympiads (see, e.g.,~\cite{cazenave2005phantom,teytaud2011lemmas,wang2015belief,zhu2015static}),
or \textit{cops-and-robber games} with limited visibility (see, e.g.,~\cite{basic2025limited,bonato2023one,clarke2020limited,dereniowski2015zero,kehagias2013cops,xue2024zero}). While there is already much research on these and similar games, so far much less is known about phantom variants of Maker-Breaker positional games.

Maker-Breaker positional games are perfect information deterministic combinatorial games. Given two positive integers $a$ and $b$, a finite set $X$, and a family $\cF\subseteq 2^X$, the $(a:b)$ Maker-Breaker game with hypergraph $(X,\mathcal{F})$ is played by two players, \emph{Maker} and \emph{Breaker}, who take turns in claiming previously unclaimed elements of $X$, with Maker going first. In each round, first Maker claims $a$ elements and then Breaker claims $b$ elements. Maker wins the game if she claims all the elements of some $F \in \cF$, and Breaker wins otherwise. The set $X$ is referred to as the \emph{board} of the game, and $\cF$ is called the family of \emph{winning sets}, while the natural numbers $a$ and $b$ are the \emph{biases} of Maker and Breaker. When $a=b=1$, the game is called \emph{unbiased}, while in other cases it is called \emph{biased}. 

Maker-Breaker games cover a wide range of games from the popular commercial games \emph{Bridge-It} and \emph{Hex} to more abstract games played on graphs. The research on Maker-Breaker games dates back to the papers of Erd\H{o}s and Selfridge~\cite{erdos1973combinatorial} and Chv\'atal and Erd\H{o}s~\cite{chvatal1978biased}, but the thorough study of this type of games has started with Beck in the late 90's and was further continued by many authors; for a systematic insight we refer the reader to the monographs~\cite{beck2008combinatorial,hefetz2014positional}.

A substantial amount of research has been dedicated to Maker-Breaker games played on the edge set of a given graph $G$, i.e.~on the board $X = E(G)$, where the first and the most commonly looked at graph is the complete graph on $n$ vertices, $K_n$. In this paper we also study games played on $K_n$. 
As for the winning sets, we focus on the four classical families that are central to the study of positional games on graphs. In particular, we consider the \emph{connectivity game}, where the winning sets are all spanning trees of $K_n$, the \emph{perfect matching game}, where the winning sets are all perfect matchings of $K_n$, the \emph{Hamiltonicity game}, where the winning sets are all Hamilton cycles of $K_n$, and the \emph{mindegree-$k$} game, where the winning sets are all spanning subgraphs of $K_n$ with minimum degree $k$.  

As $n$ increases, in all the aforementioned unbiased games on $K_n$ Maker easily wins. In order to even out the odds for Breaker and make studying of this type of games more interesting, \emph{biased games} were introduced by Chv\'{a}tal and Erd\H{o}s in their seminal paper~\cite{chvatal1978biased}, and they have been extensively studied ever since. The general goal is to determine the unique \emph{threshold bias} $b^*=b^*(a)$, i.e.~the smallest integer such that in the $(a:b)$ Maker-Breaker game Maker wins if and only if $b<b^*$. 
Consider, for instance, the $H$-game on $K_n$
with biases $(1:b)$ in which Maker wins if and only if she claims all edges of a copy of a fixed graph $H$ of constant size.
For this particular game, Bednarska and {\L}uczak~\cite{bednarska2000biased} showed that, if $H$ contains a cycle, the threshold $b^*=b^*(1)$ is of the order $n^{1/m_2(H)}$, where $m_2(H)$ denotes the maximum 2-density of $H$. While this result is interesting in itself, it shows a deep connection to Ramsey properties of the Binomial random graph. Moreover, we still find it surprising that Maker's strategy given in~\cite{bednarska2000biased}, that is optimal up to a leading constant of the bias, is random and fully ignores Breaker's edge choices. 
Hence, the paper of Bednarska and {\L}uczak is actually the first instance of a Maker-Breaker game with a phantom player, as Breaker is treated as invisible, i.e.~playing as a phantom.

Other positional games on graphs for which the threshold biases were determined asymptotically, yet no phantom variants were considered in the literature, are the aforementioned games on $K_n$. For the moment, let $a=1$. In this case, Chv\'{a}tal and Erd\H{o}s~\cite{chvatal1978biased} showed that Breaker can isolate a vertex in Maker's graph when $b=(1+\eps)\frac{n}{\log n}$, and established the threshold bias for the connectivity game to be $\Theta(\frac{n}{\log n})$. Later, Gebauer and Szab\'{o}~\cite{gebauer2009asymptotic} proved that for the connectivity and mindegree-$k$ game (for a constant $k$), the threshold bias is $(1-o(1))\frac{n}{\log n}$. Moreover, after a series of papers~\cite{beck1985random,bollobas1982biased,krivelevich2008biased},  Krivelevich~\cite{krivelevich2011critical} showed that the threshold bias for the Hamiltonicity game, and thus also for the perfect matching game, is of the same order. 

If we assume that, instead of playing optimally, both Maker and Breaker claim random unclaimed edges in the $(1:b)$ game on $K_n$, we call the players \emph{RandomMaker} and \emph{RandomBreaker}. Then the graph obtained by RandomMaker at the end of such a game is actually a random graph with $n$ vertices and $m$ edges, $\cG(n,m)$, for $m=\lceil\binom{n}{2}/{(b+1)}\rceil\sim\frac{n^2}{2b}$. From the theory of random graphs (see, e.g.,~\cite{bollobas1998random, janson2011random}) it is known that RandomMaker a.a.s.~wins the connectivity game, perfect matching game, Hamiltonicity game, and mindegree-$k$ game if and only if $b\leq (1-o(1))\frac{n}{\log n}$
(where the low order terms contained in $o(1)$ differ for these four games).
Thus, the threshold biases for these randomly played games have the same leading term as the thresholds mentioned above for the games played by clever players. This amazing relation between positional games and random graphs is attributed to Erd\H{o}s and it is known as the \emph{probabilistic intuition}. It inspired extensive research in positional game theory, particularly in determining the threshold bias for specific games and assessing whether a probabilistic intuition holds, see, e.g.,~\cite{beck1994deterministic,clemens2016random, gebauer2009asymptotic, stojakovic2005positional}.
Analogous results for reasonable values $a>1$ were shown in~\cite{fouadi2024asymptotically,hefetz2012doubly} for the connectivity game, and in~\cite{fouadi2024asymptotically,mikalavcki2013positional} for the Hamiltonicity and therefore also the perfect matching game, and in~\cite{fouadi2024asymptotically} for the mindegree-$k$ game. Moreover, also inspired by the result of Bednarska and {\L}uczak, thresholds of the so-called half-random games were determined, in which one of the two players plays fully at random while the opponent plays cleverly, see, e.g.,~\cite{groschwitz2016sharp,groschwitz2017sharp,krivelevich2015random}. 

Motivated by all the mentioned versions of the Maker-Breaker games and the phantom versions of combinatorial games as described earlier, but also inspired by real-life situations in networks, where security is concerned but the attacker's strategy is unknown, 
we look at the following one-sided imperfect information games, a phantom version of the Maker-Breaker games that we call \emph{Maker-PhantomBreaker games}.
These games are played as the usual Maker-Breaker games with the following twist. While PhantomBreaker (playing in the role of Breaker) has complete information about the whole game, Maker only gets to know about the moves of PhantomBreaker when she tries to claim some element of the board and it then turns out to be claimed by PhantomBreaker already. Note that here we must choose between two approaches: treating such Maker’s moves as failures that result in a lost move, as in~\cite{bednarska2000biased}, or allowing Maker to play again, as in, say, Phantom Tic-Tac-Toe. We opt for the former.

We analyse the games of mindegree-$k$, connectivity, perfect matching and Hamiltonicity, for any constant biases $a$ and $b$. For each of the players we describe randomized strategies and estimate the probability with which a given strategy is successful. In particular, we manage to determine all the values of the fraction $\frac{b}{a}$ for which Maker has a strategy that wins against PhantomBreaker a.a.s.

\medskip

We start with the mindegree-$k$ game. In the case $k=1$ there is a PhantomBreaker's strategy that allows a rather precise estimate on the probability of Maker's win. Note that such strategy is valuable as, similarly to the case of classical positional games on graphs, PhantomBreaker can win in both the connectivity game and the perfect matching game by applying a mindegree-$1$ game winning strategy.

\begin{theorem}\label{thm:breaker_degree1}
Let $a,b$ be positive integers with $b>2a$.
In the $(a:b)$ mindegree-1 Maker-PhantomBreaker game on $E(K_n)$, Maker's probability to win is bounded from above by
$$ \prod_{i=0}^{\left\lfloor\frac{b}{2a}\right\rfloor -1} \frac{1}{\frac{b}{2a}-i}.$$
\end{theorem}

When $\frac{b}{2a}$ is an integer, the bound in the previous theorem is $\left( (\frac{b}{2a})!\right)^{-1}$. We conjecture that this bound is tight; see the concluding remarks.

When $k>1$ we prove the following theorem, again by giving a randomized strategy for PhantomBreaker that wins with at least a constant probability. Again, this strategy for $k=2$ can clearly be applied as a winning strategy in the Hamiltonicity game.

\begin{theorem}\label{thm:breaker_degreek}
Let $a,b,k$ be positive integers with $b>\frac{2a}{k}$.
Then, PhantomBreaker has a randomized strategy 
to win the 
$(a:b)$ mindegree-$k$-game on $K_n$ against Maker
with probability at least
a constant, depending on $a$ and $b$.
\end{theorem}

We then turn our attention to Maker's strategies. The second part of the following result implies that the bounds on $b$ in the previous two theorems are best possible.

\begin{theorem} \label{thm:maker_degree}
Let $a,b,k$ be positive integers and $n$ large enough. Maker has a randomized strategy to win the $(a:b)$ mindegree-$k$-game on $K_n$ against PhantomBreaker
\begin{enumerate}[label=(\alph*)]
    \item \label{m_deg_large_b} with probability at least a constant, depending on $a$ and $b$, when $b > \frac{2a}{k}$, and
    \item \label{m_deg_small_b} a.a.s., when $b\leq \frac{2a}{k}$.
\end{enumerate}
\end{theorem}

Next, as it has already been observed for the classical Maker-Breaker games as well as their RandomMaker-RandomBreaker version and their half-random versions, it turns out that the ``bottleneck'' for occupying a perfect matching or a spanning tree is winning the mindegree-1 game as the thresholds coincide.

\begin{theorem}\label{thm:maker_matching}
Let $a,b$ be positive integers and let $n$ be large enough. Maker has a randomized strategy to win the $(a:b)$  perfect matching game on $K_n$ against PhantomBreaker
\begin{enumerate}[label=(\alph*)]
    \item \label{match_large_b} with probability at least
a constant, depending on $a$ and $b$, when $b > 2a$, and
    \item \label{match_small_b} a.a.s., when $b\leq 2a$.
\end{enumerate}

\end{theorem}

\begin{theorem} \label{thm:maker_connectivity}
Let $a,b$ be positive integers and $n$ large enough. Maker has a randomized strategy to win the $(a:b)$ connectivity game on $K_n$ against PhantomBreaker
\begin{enumerate}[label =(\alph*)]
    \item \label{conn_large_b} with probability at least a constant, depending on $a$ and $b$, when $b> 2a$, and
    \item \label{conn_small_b} a.a.s., when $b\leq 2a$.
\end{enumerate}
\end{theorem}

Finally, and also similarly 
to the already discussed variants
of Maker-Breaker games, it turns out that, in terms of the threshold bias ratio,
the ``bottleneck'' for occupying a Hamilton cycle is winning the mindegree-2 game. 

\begin{theorem} \label{thm:maker_hamilton}
Let $a,b$ be constants and $n$ large enough. 
Maker has a randomized strategy to win the 
$(a:b)$ Hamiltonicity game on $K_n$ against PhantomBreaker
\begin{enumerate}[label=(\alph*)]
    \item \label{ham_large_b} with probability at least
a constant, depending on $a$ and $b$, when $b > a$, and 
    \item \label{ham_small_b} a.a.s., when $b\leq a$.
\end{enumerate}
\end{theorem}

\bigskip

\noindent\textbf{Organization.}
In Section~\ref{sec:prelim}, we present some probabilistic tools that are used repeatedly in later sections. In Section~\ref{sec:breaker.strategies} we analyse strategies
for PhantomBreaker, thus proving Theorem~\ref{thm:breaker_degree1} and Theorem~\ref{thm:breaker_degreek}.
Finally, in Section~\ref{sec:maker.strategies} we analyse the Maker's strategies, and prove all the remaining theorems. We finish with some concluding remarks in Section~\ref{sec:concluding}.

\bigskip

\noindent\textbf{Notation.}
In a game on $K_n$ that is in progress, we let $M$, $B$, and $F$ denote the subgraphs of $K_n$ on the same vertex set consisting of Maker's, Breaker's, and the unclaimed edges, respectively.
For any graph $G\in\{K_n,M,B,F\}$, any vertex $v$ and any subsets $X,Y$ of the vertex set, we use the following notation:
$V(G)$ and $E(G)$ are the vertex set and the edge set of $G$, and $e(G):=|E(G)|$.
We write $E_G(X,Y)$ for the set of edges between $X$ and $Y$ in $G$, and set $e_G(X,Y)=|E_G(X,Y)|$,
$d_G(v,X)=e_G(\{v\},X)$, and $d_G(v)=d_G(v,V(G))$. Moreover, $G[X]$ denotes the subgraph of $G$ induced
on the vertex set $X$. In the case when $G=K_n$ we usually neglect the subscript $G$ in the above notation.

In most of our proofs we do not intend to optimize constants. Whenever these constants are not crucial we may also use the standard Landau notation. That is, given functions $f,g:\mathbb{N}\rightarrow \mathbb{R}$,
we write $f(n)=o(g(n))$ if $\lim_{n\rightarrow\infty} \left| f(n)/g(n) \right|=0$,
and $f(n)=\omega(g(n))$ if $\lim_{n\rightarrow\infty} \left| f(n)/g(n) \right| =\infty$.
Moreover, we write
$f(n)=O(g(n))$ if $\limsup_{n\rightarrow\infty} \left| f(n)/g(n) \right|<\infty$,
and $f(n)=\Omega(g(n))$ if $\liminf_{n\rightarrow\infty} \left| f(n)/g(n) \right| > 0$,
and we write $f(n)=\Theta(g(n))$ if both of the previous conditions are satisfied.
If an event, depending on $n\in\mathbb{N}$,
holds with probability tending to $1$
as $n$ tends to infinity, we say that the event holds asymptotically almost surely (a.a.s.).

\bigskip


\section{Preliminaries}\label{sec:prelim}

We make use of the following standard probability bounds, see e.g.~\cite{janson2011random}.

\begin{lemma}[Markov inequality]\label{lem:markov}
Let $X\geq 0$ be a random variable. For every $t > 0$ it holds that
$$
\Prob\left( X\geq t \right)
\leq \frac{\mathbb{E}(X)}{t}~ .
$$
\end{lemma}

\begin{lemma}[Chernoff inequality]\label{lemma:Chernoff_standard}
If $X \sim Bin(n,p)$, then
\begin{itemize}
    \item[(a)] $\Prob[X<(1-\delta)np]<\exp\left(-\frac{\delta^2np}{2}\right)$ for every $\delta>0.$
    \item[(b)] $\Prob[X>(1+\delta)np]<\exp\left(-\frac{np}{3}\right)$ for every $\delta\geq 1.$
    \item[(c)] $\Prob[X>k]<\exp\left(-k\right)$ for every $k\geq 7np$.
\end{itemize}
\end{lemma} 

Moreover, we make use of the following simple modification of the above Chernoff inequalities.

\begin{lemma}\label{lemma:Chernoff_modified}
Let $X_1,X_2,\ldots,X_n$ be a sequence of
$n$ Bernoulli random variables 
$X_i\sim Ber(p_i)$ where the $i$-th 
Bernoulli experiment is performed
after and hence independently of
the previous experiment, but where
$p_i$ may depend on the outcome
of these previously performed Bernoulli experiments,
for every $i\in [n]$.
Let $X = X_1+X_2+\ldots+X_n$.
\begin{itemize}
    \item[(a)] If $p_i\geq p$ for every $i\in [n]$, then
    $\Prob[X<(1-\delta)np]<\exp\left(-\frac{\delta^2np}{2}\right)$ for every $\delta>0.$
    \item[(b)] If $p_i\leq p$ for every $i\in [n]$, then 
    $\Prob[X>(1+\delta)np]<\exp\left(-\frac{np}{3}\right)$ for every $\delta\geq 1.$
    \item[(c)] If $p_i\leq p$ for every $i\in [n]$, then 
    $\Prob[X>k]<\exp\left(-k\right)$ for every $k\geq 7np$.
\end{itemize}
\end{lemma}

\begin{proof}
The lemma follows from Lemma~\ref{lemma:Chernoff_standard}
by a standard coupling argument. As the proofs of (a), (b) and (c) are very similar, we only give the details for (a).
While the $n$ Bernoulli experiments are preformed,
we additionally define random variables $Y_i$
for every $i\in [n]$ as follows:
If the $i$-th Bernoulli experiment is a failure, i.e.~$X_i=0$,
then we set $Y_i=0$. Otherwise, when $X_i=1$,
we let $Y_i=1$ with probability $p/p_i$ and
$Y_i=0$ with probability $1-p/p_i$.
This way, we obtain $X_i\geq Y_i$ 
and $\Prob(Y_i=1)=p$ for every $i\in [n]$, and hence
$X\geq Y_1+Y_2+\ldots+Y_n =:Y \sim Bin(n,p)$.
In particular, using Lemma~\ref{lemma:Chernoff_standard},
we obtain
$
\Prob[X<(1-\delta)np]
\leq
\Prob[Y<(1-\delta)np]
<
\exp(-\delta^2np/2)
$.
\end{proof}

\bigskip


\section{Breaker's strategies}\label{sec:breaker.strategies}

\subsection{Mindegree-1-game}

\begin{proof}[Proof of Theorem~\ref{thm:breaker_degree1}]
Phantom-Breaker plays according to the following strategy, consisting of phases: At the beginning of each phase, he chooses a vertex $v \in V(K_n)$ uniformly at random from all the vertices that are untouched by Maker at that point. Then he repeatedly claims unclaimed edges incident to that vertex. If at any point Maker tries to claim an edge incident to $v$ before Breaker claimed all edges incident to $v$, Breaker finishes with this phase and starts a new one. If at the beginning of a phase there are no more vertices untouched by Maker, then Breaker forfeits the game.

\smallskip

\textbf{Strategy discussion:}
We refer to Breaker's phase as \emph{successful} if he manages to claim a full star at $v$ and win the game. Otherwise, a phase is \emph{unsuccessful}. Note that each unsuccessful phase lasts for at most $\lfloor \frac{n-2}{b} \rfloor$ rounds. Just before Breaker's $i$-th phase, the number of vertices untouched by Maker is at least $n - (i-1)\cdot 2a\lfloor \frac{n-2}{b}\rfloor$, and during the $i$-th phase Maker touches at most $2a\lfloor \frac{n-1}{b}\rfloor$ vertices. Hence, conditioned on the previous $i-1$ phases being unsuccessful, the $i$-th phase is unsuccessful with probability upper bounded by
$$\frac{2a\lfloor \frac{n-1}{b}\rfloor}{n - (i-1)\cdot 2a\lfloor \frac{n-2}{b}\rfloor} \leq  \frac{1}{\frac{b}{2a}-(i-1)}. $$

Clearly, Breaker's phases have a positive probability of being successful as long as the previous upper bound is less than 1. Hence, by law of total probability, the probability that all those phases are unsuccessful is upper bounded by
$$ \prod_{i=1}^{\left\lfloor\frac{b}{2a}\right\rfloor} \frac{1}{\frac{b}{2a}-(i-1)}.$$
\end{proof}

\subsection{Mindegree-$k$-game}

\begin{proof}[Proof of Theorem~\ref{thm:breaker_degreek}]
PhantomBreaker's strategy is simple, he chooses a vertex $v \in V(K_n)$ uniformly at random from all $n$ vertices. Then he repeatedly claims unclaimed edges incident to $v$ until there are less than $k$ such edges left. If Maker claims $k$ edges incident to $v$, Breaker forfeits the game. Also, if at any point Maker tries to claim an edge incident to $v$ that Breaker already claimed, Breaker forfeits the game.

\smallskip

\textbf{Strategy discussion:}
Obviously, if Breaker wins he does so within the first $x:=\lceil \frac{n-k}{b} \rceil$ moves, and only if Maker has not attempted to claim any Breaker's edges during that time. Under that assumption, both players play independent of the opponents moves.
Therefore, we can look at both Maker's graph and Breaker's graph after $x$ moves under the assumption that Maker has not attempted to claim any Breaker's edges, concluding that Breaker wins if those two graphs do not intersect, and Maker's degree of $v$ did not reach $k$.

As $\frac{a}{b}<\frac{k}{2}$, we can find a constant $\eps>0$ be such that 
$\frac{a}{b}<(1-\eps)\frac{k}{2}$. After $x$ moves, the sum of degrees in Maker's graph is $2ax <(1-\eps)kn$. Hence, at least $\eps n$ vertices in Maker's graph have not reached the degree $k$.
Furthermore, during the last $y:=\frac{\eps n}{4a}$ of the first $x$ Maker's moves, Maker touched at most $2a \cdot \frac{\eps n}{4a} = \frac{\eps n}{2}$ vertices. Therefore, with probability at least $\frac{\eps}{2}$, the vertex $v$ is among the vertices that in Maker's graph have not reached degree $k$ after $x$ moves, and Maker has not touched them within the last $y$ moves. We proceed assuming that this happened.

This means that every time Maker claimed an edge incident to $v$, there were at least $by$ free edges at $v$. Knowing that at any point of the game Breaker's edges are a random subset of all the edges incident to $v$, the probability for Maker to claim one of the unclaimed edges is at least $\delta:=\frac{by}{n} = \frac{\eps b}{4a}$. Therefore, the probability for Maker to claim an unclaimed edge in every of the at most $k-1$ attempts to claim an edge incident to $v$ is at least $\delta^{k-1}$.

Putting everything together, the probability of Breaker's win is at least $\frac{1}{2}\eps \delta^{k-1}$.
\end{proof}

\bigskip


\section{Maker's strategies}
\label{sec:maker.strategies}

In the following subsections we are going to prove
the Theorems~\ref{thm:maker_degree}--\ref{thm:maker_hamilton}.
For each of the strategies described in their proofs
we do the following: 
Even though Maker should attempt to claim $a$ edges per round, 
we always give a sequence of instructions for Maker to play single edges. When it is her turn, she simply follows the instructions until she plays $a$ edges. For the next round, she simply continues
following these instructions from the point where she stopped before.
Whenever Maker tries to claim an edge which is already occupied by PhantomBreaker, we call her choice a failure.

Additionally, sometimes when we analyse Maker's strategy, it turns out to be useful to split the 
game into phases of a small linear number of moves or rounds. We then usually define a parameter $\varepsilon$,
depending on the biases $a$ and $b$,
and define the length of such a phase to be $\varepsilon n$. To avoid
using floor and ceiling signs,
we then assume $\frac {1}{\eps}$ (number of phases) and $\eps n$ (length of phases) to be integers, instead of writing 
$\lceil \frac{1}{\eps} \rceil$ and $\lceil \eps n \rceil$
all the time, as this rounding is never crucial for our calculations.


\subsection{Mindegree-$k$-game with $b> \frac{2a}{k}$}

\begin{proof}[Proof of Theorem~\ref{thm:maker_degree}~\ref{m_deg_large_b}]
    Let $\varepsilon=\frac{a}{20kb}$.
	Maker plays the following simple randomized strategy:	
	Initially, let $V_0=V(K_n)$. As long as $V_0\neq \varnothing$,
	choose a vertex $v$ from $V_0$ uniformly at random, then claim edges incident to $v$ uniformly at random 
	until $d_M(v)\geq k$, and then remove $v$ from $V_0$.
			
	Let $v_1,v_2,v_3,\dots$ be the order in which the vertices 
	are chosen by Maker.	
	We split Maker's strategy into phases
	where the $i$-th phase stops after 
	Maker chooses $v_{i \varepsilon n }$, 
	for $1\leq i\leq \frac{1}{\varepsilon}$. 
	Maker forfeits the game in case one of these phases lasts for more than 
	$\frac{2k\varepsilon n}{a}$ rounds, i.e.~if Maker does not manage to increase
	the degree of all vertices $v_{(i-1)\varepsilon n + 1 },\ldots,
		v_{i \varepsilon n }$ to at least $k$
	within $\frac{2k\varepsilon n}{a}$ rounds.

\medskip

{\bf Strategy discussion:}
Our goal is to show that Maker never forfeits the game with
probability at least $\frac{1}{(20k\cdot b/a)^{16k\cdot b/a}}$.
In order to analyse her strategy,  we consider the following events:
\begin{itemize}
\item event $\mathcal{A}_i$: Maker does not forfeit during the $i$-th phase, and this phase lasts at most $\frac{k\varepsilon n}{a} + \sqrt{n}\log^3 n$ rounds,
\item event $\mathcal{B}_i$: let $D_{i-1}$ be the vertices
whose Breaker degree increased to $\frac{n}{4}$ during phase $i-1$, then all of $D_{i-1}$ are chosen until the end of the $i$-th phase.
\end{itemize}
We also set $\mathcal{E}_i = \bigwedge_{j \leq i} (\mathcal{A}_j \land \mathcal{B}_j)$.
Moreover, note that by the bound on the number of rounds for each phase,
we have that the total number of rounds is upper bounded by $\frac{2kn}{a}$. We first prove the following two claims.

\smallskip

\begin{claim}
For every $i\leq \frac{1}{\varepsilon}$ we have $\operatorname{Pr}(\mathcal{A}_i|\mathcal{E}_{i-1}) = 1-o(1)$.
\end{claim}

\begin{proof}
For this proof, call a vertex \textit{bad} if its Breaker degree is at least $\sqrt{n}$.
Since the whole game lasts $O(n)$ rounds,
during the $i$-th phase at most $O(\sqrt{n})$ of the vertices $v_{(i-1)\varepsilon n +1},\ldots,v_{i \varepsilon n}$ can be bad at the moment when they are chosen or become bad after they are chosen.
However, even for these vertices $v_j$ we know that
$d_B(v_j)< \frac{n}{2}$ throughout phase $i$,
by conditioning on $\mathcal{B}_{i-1}$ and since the phases
$i-1$ and $i$ together 
last at most $\frac{4k\varepsilon n}{a}$ rounds in which
Breaker can claim at most $\frac{n}{4}$ edges. Thus, whenever Maker tries to claim an edge incident to a bad vertex $v_j$,
the probability of success is at least $\frac{1}{2}$.
Therefore, applying Lemma~\ref{lemma:Chernoff_modified}(a)
with $p=\frac{1}{2}$, together with a union bound,
we see that a.a.s.~for each of these bad vertices we need to make $O(\log^2 n)$ random choices of incident edges until getting that $d_M(v_j)\geq k$. 
For all the other vertices $v_j$ chosen during the $i$-th phase, it always holds that $d_B(v_j)<\sqrt{n}$ throughout phase $i$, and hence the choice of a random edge always fails with probability smaller than 
$\frac{2}{\sqrt{n}}$.
Applying Lemma~\ref{lemma:Chernoff_modified}(b)
with $p=\frac{2}{\sqrt{n}}$,
together with a union bound, we get that a.a.s.~for all these vertices a total of
$O(\sqrt{n})$ failures happen.

As in total, at most $k\varepsilon n$ edges need to be claimed in
the $i$-th phase while Maker's bias is $a$,
and since a.a.s.~$O(\sqrt{n}\log^2 n)$ failures happen
throughout the phase, we get that $\mathcal{A}_i$
follows a.a.s.
\end{proof}

\begin{claim} \label{c:prob-bj-min-k}
For every $i\leq \frac{1}{\varepsilon}$ we have $\operatorname{Pr}(\mathcal{B}_i|\mathcal{E}_{i-1} \land \mathcal{A}_i) \geq (1-o(1)) \varepsilon^{|V_{i-1}|}$.
\end{claim}

\begin{proof}
Note that the claim is obvious for the last phase,
since at the end of this phase every vertex is chosen.
So, we only need to consider earlier phases $i$.
Let $D_{i-1}^\ast \subset D_{i-1}$ 
be the subset of vertices in
$D_{i-1}$ that were not chosen by Maker until the end of phase $i-1$; set $t:=|D_{i-1}^\ast|$
and note that $t=O(1)$ by our general bound on the number of rounds.
Since the set of all
the vertices chosen during the $i$-th phase
forms a set of size $\varepsilon n $ 
chosen uniformly at random
from the set $V\setminus \{v_j:~ j\leq (i-1) \varepsilon n\}$ we obtain
\begin{align*}
\Prob(\text{all of $D_{i-1}^\ast$ are chosen in the $i$-th phase})
& =
\frac{\binom{n-(i-1)\varepsilon n-t}{\varepsilon n-t}}{\binom{n-(i-1)\varepsilon n}{\varepsilon n}}
=
\prod_{j=0}^{t-1}
\frac{\varepsilon n - j}{n-(i-1)\varepsilon n - j} \\
& \geq
\left( \frac{\varepsilon n-t+1}{n} \right)^t
\geq 
(1-o(1))\varepsilon^{|D_{i-1}|} ,
\end{align*}
which proves the claim.
\end{proof}

Now, using the above claims, we can conclude that
\begin{align*}
\Prob(\text{Maker wins})
& \geq
\Prob\left(  \mathcal{E}_{1/\varepsilon} \right)
=
\prod_{i=1}^{1/\varepsilon} 
\Prob\left( \mathcal{E}_i | \mathcal{E}_{i-1} \right)
=
\prod_{i=1}^{1/\varepsilon}
\Prob\left( \mathcal{A}_i | \mathcal{E}_{i-1} \right)\cdot
\Prob\left( \mathcal{B}_i | \mathcal{E}_{i-1} \land \mathcal{A}_i \right) \\
& \geq
\prod_{i=1}^{1/\varepsilon}
(1-o(1))\varepsilon^{|D_{i-1}|}
=
(1-o(1))\varepsilon^{\sum_{i=1}^{1/\varepsilon} |D_{i-1}|}
> 
\left( \frac{1}{20k\cdot b/a} \right)^{16k\cdot b/a},
\end{align*}
where
the last inequality uses the definition of $\varepsilon$
and the fact that Breaker cannot create 
$16k\cdot b/a$ vertices of degree at least $\frac{n}{4}$
within $\frac{2kn}{a}$ rounds.
This finishes the proof of the theorem.
\end{proof}

\bigskip


\subsection{Mindegree-$k$-game with $b\leq \frac{2a}{k}$}

\begin{proof}[Proof of Theorem~\ref{thm:maker_degree}~\ref{m_deg_small_b}]
We may assume that $b=\lfloor \frac{2a}{k} \rfloor$,
since it is not a disadvantage for Maker to play against 
a smaller bias $b$. Let $\varepsilon = (10a)^{-2}$.
Before the game starts, we set $\omega(v)=0$
for every vertex $v\in V(K_n)$. This value will
increase during the game.
Moreover, at any moment in the game we let
$V_{<k} := \{v\in V(K_n):~ \omega(v)<k\}.$
Maker's strategy proceeds in two stages.	

\medskip
	
	{\bf Stage I:} As long as $V_{<k}$ has more than 
	$\frac{n}{\log n}$ vertices, Maker repeatedly does the 
	following:
	
	\begin{itemize}
	\item[(1)] Maker chooses two distinct 
	vertices $v,w \in V_{<k}$ uniformly
	at random and tries to claim the edge between them. 
	Two cases may happen:	
	\begin{itemize}
	\item[(a)] 
	If she succeeds, she claims the edge.
	\item[(b)] 
	If she fails, she proceeds to randomly claim edges 
	incident to $v$ until $d_M(v)\geq k$, 
	and afterwards she does the same for $w$ until $d_M(w)\geq k$. 
	Note that this degree condition
	may already hold before $v,w$ were chosen. In this
	case, Maker does not need to claim any edges.
	\end{itemize}
	\item[(2)] In any case, Maker increases 
	$\omega(v)$ and $\omega(w)$ by just 1, 
	and updates $V_{<k}$ accordingly. 
	
	Afterwards, she starts the next iteration, 
	except if she already tried to claim more than 
	$\log^{10}(n)$ edges according to (b). In 
	the latter case she forfeits the game.
	\end{itemize}

    Once $V_{<k}$ has size at most $\frac{n}{\log n}$, Maker 
    proceeds to Stage II.
    
    \medskip
	
	{\bf Stage II:} As long as $V_{<k}\neq \varnothing$, 
	Maker repeatedly does the following:
	
	Maker chooses a vertex $v\in V_{<k}$ uniformly at random,
	and tries to claim incident edges uniformly at random 
	until $d_M(v)\geq k$. Then she removes $v$ from $V_{<k}$.	
    She repeats this until $V_{<k}$ is empty.	
    If this stage lasts longer than $\varepsilon n$ rounds, 
    she forfeits the game.
	
	\bigskip
	
	{\bf Strategy discussion:} We need to show that Maker 
	a.a.s.~wins the game if she follows the proposed strategy. 
	As a first step, we state and prove the following two 
	simple claims.
	 	
	\begin{claim}\label{claim:omega-degree}
	At any moment in Stage I, 
	as long as Maker does not forfeit the game,
	$d_M(v)\geq \omega(v)$ for every $v\in V(K_n)$.	
	\end{claim}
	
	\begin{proof}
	Assume that a vertex $v$ is chosen in step (1)
	of some iteration and hence $\omega(v)<k$
	at that point. 
	Then, if case (a) happens, both $d_M(v)$ and 
	$\omega(v)$ increase by 1.
	If instead (b) happens, Maker makes sure that
	$d_M(v)\geq k$ holds at the end of the iteration,
	while $\omega(v)$ is increased by 1 only.
	Otherwise it may happen that $d_M(v)$
	is increased by 1 if (b) is performed for some 
	other vertex $x$ and Maker claims $xv$,
	but in this case $\omega(v)$ stays unchanged.
	Therefore, the claim follows by induction.
	\end{proof}
	
	\begin{claim}\label{clm:min-deg.process}
	There exists a constant $\alpha>0$ such that the following 
	holds a.a.s.:~as long as 
	$|V_{<k}|\geq \frac{\varepsilon n}{2}$, 
	there are at least $\alpha n$ vertices $v$ with
	$\omega(v)=0$.
	\end{claim}

\begin{proof}
Consider the following two random processes $R$ and $R'$:

\begin{itemize}
\item Process $R$:
as long as $|V_{<k}|\geq \frac{\varepsilon n}{2}$ holds,
choose a pair $v,w$ uniformly at random from
$V_{<k}$, increase $\omega(v)$ and $\omega(w)$ by exactly 1; update $V_{<k}$ afterwards.

\item Process $R'$:
as long as $|V_{<k}|\geq \frac{\varepsilon n}{2}$ holds,
choose a pair $v,w$ uniformly at random from
$V$. If $v$ or $w$ does not belong to $V_{<k}$, then ignore this choice. Otherwise, i.e.~if both vertices belong to $V_{<k}$, increase $\omega(v)$ and $\omega(w)$ by exactly 1; and update $V_{<k}$ afterwards.
\end{itemize} 

Process $R$ is part of Maker's strategy in Stage I.
However, Maker could also choose to use $R'$
instead, and whenever $\{v,w\}\not\subset V_{<k}$
happens she just needs to ignore this choice for her strategy
and make a new random choice.
Now, as long as $|V_{<k}|\geq \frac{\varepsilon n}{2}$, 
we get that each time a random choice $\{v,w\}$
is done in the process $R'$, 
the probability that it belongs to $V_{<k}$ is at least 
$(\frac{\varepsilon}{2})^2$.
Moreover, each time such a choice belongs to 
$V_{<k}$, both values $\omega(v)$ and $\omega(w)$
are increased by 1; and hence it requires less than
$\frac{kn}{2}$ such successful choices 
$\{v,w\} \subset V_{<k}$ until 
$|V_{<k}| < \frac{\varepsilon n}{2}$ holds.
Now, by applying Lemma~\ref{lemma:Chernoff_modified}(a) with $p=(\frac{\varepsilon}{2})^2$,
we get that a.a.s~at most $\frac{4kn}{\varepsilon^2}$ 
iterations of $R'$ are enough
to reach that point when $|V_{<k}| < \frac{\varepsilon n}{2}$.
These $\frac{4kn}{\varepsilon^2}$ edges form a random graph 
$G(n,\frac{4kn}{\varepsilon^2})$ chosen uniformly among 
all graphs on $n$ vertices with $\frac{4kn}{\varepsilon^2}$ edges.
It is well known that such a random graph a.a.s.~has at least 
$\alpha n$ isolated vertices,
for an appropriately chosen constant 
$\alpha=\alpha(k,\varepsilon)$
(see e.g.~\cite{bollobas1998random}, or Theorem 3.3 (with $d=0$) together with Lemma 1.3 in~\cite{frieze2015introduction}).
\end{proof}

Next, let us take a closer look at Maker's randomized strategy and bound the number of failures in Stage I.

	\begin{claim}
	The following a.a.s.~holds: If Maker does not forfeit 
	in Stage I, there are at most $O(\log^4(n))$ 
	failures in Stage~I.
	\end{claim}
	
	\begin{proof}
	Since each iteration in Stage I increases $\omega(\cdot)$
	for two vertices in $V_{<k}$ until 
	this set has size at most $\frac{n}{\log n}$,
    at most $\frac{kn}{2}-\frac{n}{2\log n}$ such iterations 
    will be performed. Assuming that Maker does never forfeit 
    the game (i.e.~she tries to claim at most $\log^{10}n$ 
    edges according to case (b)), 
    these iterations happen within
    $\frac{kn}{2a}-\frac{n}{2a\log n} + \log^{10}n	
    < \frac{kn}{2a}-\frac{n}{4a\log n}$
	rounds. In the meantime, Breaker can claim
	at most
	$n - \frac{n}{2k\log n}$
	edges.
	Therefore, whenever Maker chooses a pair 
	$\{v,w\}$ in $V_{<k}$ according to (1), 
	she has at least $\frac{n^2}{2\log^2n}$ possible choices, 
	among which less than $n$ edges $vw$ are blocked already. 
	Thus, when Maker wants to claim an edge $vw$
	according to (1), 
	the probability that $vw$ is already occupied is 
	bounded from above by $\frac{2\log^2 n}{n}$. 
	By Lemma~\ref{lemma:Chernoff_modified}(b)
	with $p=\frac{2\log^2 n}{n}$,
	this will happen a.a.s.~only $O(\log^2 n)$ 
	times among all iterations of Stage I. 
	After Maker fails to claim an edge 
	$vw \in E(K_n)$ according to (1), she 
	will claim edges incident to $v$ and $w$ until 
	both have 
	degree at least $k$. Since Breaker's degree at 
	each vertex 
	is at most $n - \frac{n}{2k\log{n}}$, 
	whenever Maker 
	tries to claim an edge, the probability 
	of success is at least
	$\frac{1}{2k\log n}$. 
	Applying Lemma~\ref{lemma:Chernoff_modified}(a)
	with $p=\frac{1}{2k\log n}$,
	it follows that a.a.s. $O(\log^2 n)$ trials 
	are enough to increase the degree of 
	$v$ and $w$ to $k$,
	even with union bound over all 
	failures made in (1). 
	Thus, Maker a.a.s.~fails at most $O(\log^4(n))$ 
	times during Stage I. 
	\end{proof}

	In the following, we condition on the event described in 
	the above claim. Note that this in particular implies that
	Maker does not forfeit the game during Stage I.
	At the end of this stage, because of 
	Claim~\ref{claim:omega-degree}, every vertex in 
	$V\setminus V_{<k}$ has degree at least $k$.
    So, if Maker can a.a.s.~follow Stage~II as described, 
    i.e.~always increase $d_M(v)$ to $k$ for some vertex $v$
    until $V_{<k}=\varnothing$, 
    then she a.a.s.~wins. To show that this is indeed the case, 
    we next prove the following.

	\begin{claim}\label{clm:mindeg:dB-bound}
		At the end of Stage I a.a.s. 
		for all $v \in V_{<k}$ we have 
		$d_B(v) \leq (\frac{1}{2} + \epsilon)n$.
	\end{claim}

\begin{proof}
	As explained in the previous proof, Breaker a.a.s.~claims 
	in total at most 
	$n - \Theta(\frac{n}{\log{n}})$ edges in Stage I, and 
	therefore can only achieve for one vertex $v \in V(K_n)$ 
	that $d_B(v) \geq \frac{1}{2}n$.
	Assume that $d_B(v) \geq (\frac{1}{2} + \epsilon)n$ at the 
	end of Stage I. Our goal is to show that at this point
	a.a.s.~$v\notin V_{<k}$.

	At the time Breaker achieved $d_B(v) \geq \frac{1}{2}n$, 
	there were still at least $\frac{\epsilon n}{b}$ rounds to 
	be played according to Stage I, 
	and hence $|V_{<k}|\geq \frac{\epsilon n}{2}$
	at that moment. By Claim~\ref{clm:min-deg.process}
	we know that a.a.s.~at that moment there are at least 
	$\alpha n$ vertices $u$ with $\omega(u)=0$, for some 
	constant $\alpha$. We choose $V_1$ to be a set of 
	$\alpha n$ such vertices, and let $y$ be any fixed vertex 
	in $V_1\setminus \{v\}$. We write $V_{<k}^{final}$ for the 
	set $V_{<k}$ at the end of Stage I, and note that every 
	vertex in $V_1$ has the same probability of ending up in
	$V_{<k}^{final}$. Moreover, since 
	$\omega(v)\geq \omega(y)$ at the described moment,
	we conclude
\begin{align*}
\text{Pr}(v\in V_{<k}^{final}) 
\leq
& \text{Pr}(y\in V_{<k}^{final}) \\
=
& \sum_{t=0}^{\frac{n}{\log n}}
\text{Pr}\left(y\in V_{<k}^{final} \Big| |V_1\cap V_{<k}^{final}|=t \right)
\cdot \text{Pr($|V_1\cap V_{<k}^{final}|=t$)} \\
=
& \sum_{t=0}^{\frac{n}{\log n}}
\frac{t}{\alpha n}
\cdot \text{Pr($|V_1\cap V_{<k}^{final}|=t$)} 
\leq \frac{1}{\alpha \log n} = o (1) .
\end{align*}
This proves the claim.
\end{proof}	

		It remains to be shown that Maker succeeds with her strategy in Stage II. For this note the following. Since by Claim~\ref{clm:mindeg:dB-bound}
    every vertex $v \in V_{<k}$ fulfils $d_B(v) \leq (\frac{1}{2} + \epsilon)n$ at the beginning of Stage II, and since Stage II lasts at most $\epsilon n$ rounds, Maker succeeds with probability at least $\frac{1}{2} - 2\epsilon b > \frac{1}{4}$ whenever she tries to claim an edge in Stage~II. 
For finishing Stage II, Maker in total needs at most 
$k\cdot \frac{n}{\log n} = \Theta(\frac{n}{\log n})$ successes, because $|V_{<k}|\leq \frac{n}{\log n}$ when Stage II starts. Applying Lemma~\ref{lemma:Chernoff_modified}(a) it follows that this number of successes a.a.s. happens within less than $\epsilon n$ rounds. Hence, Maker wins a.a.s.
\end{proof}

\bigskip


\subsection{Perfect matching game with $b> 2a$}

\begin{proof}[Proof of Theorem~\ref{thm:maker_matching}~\ref{match_large_b}]
We start by describing Maker's strategy. 
To fix an upper bound on the number of rounds, we let Maker forfeit
if she does not win the game within $\frac{n}{2a}+n^{0.99}$ rounds.
During the game, she will maintain two sets of vertices, $V_M$ -- vertices in Maker's matching, and $V_0$ -- vertices not yet in Maker's matching. Initially, $V_M$ is empty and $V_0$ contains all $n$ vertices.

\medskip

\textbf{Stage I:} This stage lasts for $\frac{n}{2} - n^{0.7}$ steps. In each step, Maker chooses u.a.r.~two vertices $x_1, x_2\in V_0$. 

\begin{enumerate}
\item[(1)] She first tries to claim $x_1x_2$. If successful, that completes the step. Then the vertices $x_1, x_2$ are matched and moved from $V_0$ to $V_M$.

\item[(2)] Otherwise, for each of the vertices $x_i$, $i\in \{1,2\}$, Maker does the following. She chooses u.a.r.~a vertex $y$ from $V\setminus \{x_i\}$ and tries to claim $y x_i$. As long as that is unsuccessful, a new $y$ is chosen the same way and again she attempts to claim $y x_i$.

Once she succeeds, if $y\in V_0$, then $x_i$ and $y$ are matched and moved from $V_0$ to $V_M$. If, however, $y\in V_M$ and it is matched to $y'\in  V_M$, then Maker tries to claim the edge between $y'$ and a vertex $z$ chosen u.a.r.~from $V_0$. If that is successful, both $x_i$ and $z$ are moved from $V_0$ to $V_M$, with $y$ now being matched to $x_i$, and $y'$ being matched to $z$, ending the step. If claiming $y'z$ is unsuccessful, a new $y$ is chosen again as above and the process is repeated.

Once this process is successful for both $x_1$ and $x_2$, the step is completed. If the process of choosing a new $y$ in (2) is repeated more than $8 \log n$ times within the same step, Maker forfeits the game.

\end{enumerate}

\textbf{Stage II:} This stage lasts until Maker either completes a perfect matching, or she forfeits the game. In each step, she chooses u.a.r.~two vertices $x_1, x_2\in V_0$. Then $2n^{0.1}$ many matched pairs $s^{(1)}_1 t^{(1)}_1$, $s^{(2)}_1 t^{(2)}_1$, $s^{(1)}_2 t^{(1)}_2$, $s^{(2)}_2 t^{(2)}_2$, \dots, $s^{(1)}_{n^{0.1}} t^{(1)}_{n^{0.1}}$, $s^{(2)}_{n^{0.1}} t^{(2)}_{n^{0.1}}$,  from $V_M$ are chosen u.a.r.~one after another. Whenever such a pair $s^{(i)}_j t^{(i)}_j$ is chosen, Maker attempts to claim the edge $x_i s^{(i)}_j$. Once done with this process, Maker repeatedly tries to claim the edge $t^{(1)}_{j'} t^{(2)}_{j''}$, where $j', j''$ are chosen u.a.r., until she successfully completes a path of length five, $x_1 s^{(1)}_{j'} t^{(1)}_{j'} t^{(2)}_{j''}s^{(2)}_{j''} x_2$. Then $x_1, x_2$ are moved from $V_0$ to $V_M$, and the six vertices are matched in pairs along this path, completing the step. If, however, after $n^{0.1}$ choices of $j', j''$ the step is still not complete, Maker forfeits the game.

\medskip

\textbf{Strategy discussion:}
Our goal is to show that Maker does never forfeit the game with probability at least 
$\frac{1}{(10b/a)^{11\cdot b/a}}$.
Let $\varepsilon = \frac{a}{10b}$.
In order to analyse her strategy, 
we split it into
$\frac{1}{\eps}$ phases, where the 
$i$-th phase finishes when $V_M$ contains at least $i\eps n$ vertices.
Moreover,
we consider the following events:
\begin{itemize}
\item event $\mathcal{A}_i$: Maker completes the $i$-th phase without forfeiting, and the phase lasts at most $\frac{\eps n}{2a} + n^{0.9}$ rounds. 
\item event $\mathcal{B}_i$: 
let $D_{i-1}$ be the vertices
whose Breaker degree increased to $\frac{n}{10}$ during phase $i-1$, then all of $D_{i-1}$ are moved to $V_M$ until the end of the $i$-th phase.
\end{itemize}

We also set $\mathcal{E}_i = \bigwedge_{j \leq i} (\mathcal{A}_j \land \mathcal{B}_j)$. 
Moreover, note that by bounding the total number of rounds with
$\frac{n}{2a} + n^{0.99}$, the total number of Breaker's edges is upper bounded by $B^{\text{(tot)}}:=(1 + o(1))\frac{bn}{2a}$ throughout the game.
If on top of that all the events $\mathcal{E}_i$, with $i \leq \frac{1}{\eps}-1$, occur, then Breaker's maximum degree in $V_0$ will never get higher than 
$(1+o(1))(\frac{1}{10} + \frac{b\eps}{a}) n < \frac{n}{4}$. To see this note that the number of rounds played within each phase is $(1+o(1))\frac{\eps}{2a} n$ because of the $\mathcal{A}_i$, and under assumption of $\mathcal{B}_i$, every vertex that reaches Breaker degree at least $\frac{n}{10}$ until or during the $(i-1)$-st phase 
is moved to $V_M$ until the end of the $i$-th phase.

We next prove a few claims similarly to the proof of Theorem~\ref{thm:maker_degree}~\ref{m_deg_large_b}.

\begin{claim} \label{c:prob-ai-pm}
For every $i<\frac{1}{\eps}$ we have $\operatorname{Pr}(\mathcal{A}_i|\, \mathcal{E}_{i-1}) = 1-o(1)$.
\end{claim}

\begin{proof}
We assume that $\mathcal{E}_{i-1}$ occurred. First note that for $i<\frac{1}{\eps}$ the $i$-th phase is still completely part of Stage~I. As throughout the $i$-th phase the number of pairs of vertices in $V_0$ is $\Theta(n^2)$ and the number of edges Breaker claimed is not more than $B^{\text{(tot)}}$, which is linear in $n$, the expected number of unsuccessful attempts to claim $x_1x_2$ in (1) is constant. Due to Markov's inequality, the probability that this number is larger than $\log n$ is upper bounded by $o(1)$.

In each of those unsuccessful attempts we proceed to (2). Knowing that $\mathcal{B}_j$ occurred for all $j<i$, the Breaker's degree of $x_1$ and $x_2$ is upper bounded by $\frac{n}{4}$
throughout the $i$-th phase. 
If $i=1$ and hence $|V_M|\leq \varepsilon n$,
we have that whenever a vertex $y$ 
is chosen according to (2),
then with probability larger than $0.5$ we have $y\in V_0$ and $yx_1$ or $yx_2$, respectively, is unclaimed. Otherwise, if $i\geq 2$, we see the following:
As there are more than $(i-1)\eps n$ vertices in $V_M$ and Breaker's total number of edges is less than $B^{\text{(tot)}}$, the average degree of a matching edge (where the degree of an edge is the sum of the degrees of its two vertices) in $V_M$ is at most $c_1 := \frac{4B^{\text{(tot)}}}{(i-1)\eps n} = O(1)$. Then less than a $0.1$ fraction of those matching edges have a degree larger than $10 c_1$. Therefore, with probability at least $1-\frac{1}{4}-0.2=0.55$, the vertex $y$ chosen according to (2) will be such that $yx_1$ or $yx_2$, respectively, is unclaimed, and either $y\in V_0$, or $y\in V_M$ and its matching edge $yy'$ has degree at most $10 c_1$. As the number of vertices in $V_0$ is $\omega(1)$, in the latter case the edge $y'z$ will be unclaimed with probability $1-o(1)$. 

All in all, for each $y$ chosen in part (2) the probability of successfully wrapping it up and completing the step is lower bounded by $0.5$. Therefore, the probability that $4 \log n$ consecutive choices of $y$ do not lead to the completion of (2) for the vertex $x_1$ or $x_2$, respectively, is upper bounded by $0.5^{4 \log n} = O(n^{-2})$.

Looking at the whole $i$-th phase, by the union bound, the probability that there are less than $\log n$ steps with unsuccessful attempts to claim $x_1x_2$ in~(1), and that each of them needs less than $8 \log n$ additional edge choices according to (2), is $1-o(1)$. Note that in the likely event Maker does at most $\frac{\eps n}{2} + 8 \log^2 n$ edge choices, and hence at most $\frac{\eps n}{2a} + \log^3n$ rounds are needed.
\end{proof}

\begin{claim} \label{c:prob-bi-pm}
For every $i<\frac{1}{\eps}$ we have $\operatorname{Pr}(\mathcal{B}_i|\, \mathcal{E}_{i-1} \land \mathcal{A}_i) \geq (1-o(1)) \varepsilon^{|D_{i-1}|}$.
\end{claim}

\begin{proof}
The vertices which are chosen in the $i$-th phase 
according to (1) and hence moved to $V_M$ by assuming 
$\mathcal{A}_i$ form a set of size $\varepsilon n$
chosen uniformly at random from the set $V_0$ as given at the beginning of the phase. Hence, the claimed inequality
can be proven exactly the same way as in 
Claim~\ref{c:prob-bj-min-k}.
\end{proof}

\begin{claim} \label{c:prob-last-ai-pm}
We have $\operatorname{Pr}(\mathcal{A}_{1/\eps}|\, \mathcal{E}_{1/\eps-1}) = 1-o(1)$.
\end{claim}

\begin{proof}
We assume that $\mathcal{E}_{1/\eps-1}$ holds.
At first we consider the first part of the last phase
consisting of all the remaining rounds played in Stage I.
Throughout this part of the game, 
the number of pairs of vertices in $V_0$ is $\Omega(n^{1.4})$ and the number of edges Breaker claimed is not more than $B^{\text{(tot)}}$, which is linear in $n$. Hence, the expected number of unsuccessful attempts to claim $x_1x_2$ in (1) is $O(n^{0.6})$. 
Due to Lemma~\ref{lemma:Chernoff_modified}(c), the probability that this number is larger than $n^{0.6}\log n$ is upper bounded by $o(1)$.

In each of those failed attempts we proceed to (2),
and we can do the same analysis as in the proof of Claim~\ref{c:prob-ai-pm} with $c_1$
being replaced with $c_2 := \frac{4B^{\text{(tot)}}}{(1 - \eps) n} = O(1)$. It then follows that
the probability that there are less than $n^{0.6}\log n$ steps with unsuccessful attempts to claim $x_1x_2$ in~(1), and that each of them needs less than $8 \log n$ additional edge choices according to (2), is $1-o(1)$. 
Moreover, in this likely event, Maker does less than
$\frac{\varepsilon n}{2}$ edge choices,
and hence plays less than $\frac{\varepsilon n}{2a}$ rounds
before reaching Stage II.

We next consider the second part of the last phase, i.e.~Stage II.
We still have that Breaker's degree of $x_1$ and $x_2$ is upper bounded by $\frac{n}{4}$, and less than a $0.1$ fraction of the matching edges in $V_M$ have a degree larger than $10 c_2$. Hence, with probability at least $1-\frac{1}{4}-0.2 = 0.55$, Maker's attempt to claim the edge $x_i s^{(i)}_j$ will be successful, while additionally the degree of the matched edge $s^{(i)}_j t^{(i)}_j$ is at most $10 c_2$.

Then, for $i \in \{1,2 \}$, the probability that out of $n^{0.1}$ matching edges $s^{(i)}_1 t^{(i)}_1$, $s^{(i)}_2 t^{(i)}_2$, \dots, $s^{(i)}_{n^{0.1}} t^{(i)}_{n^{0.1}}$ the above two conditions are satisfied for at least $\frac12 n^{0.1}$ of them 
is at least $1- \exp(- \Theta(n^{0.1}))$, by applying Lemma~\ref{lemma:Chernoff_modified}(a) with $p=0.55$.
If we assume that this holds for both $i=1,2$, the number of pairs $j', j''$ for which the above conditions hold for both $s^{(1)}_{j'} t^{(1)}_{j'}$ and $s^{(2)}_{j''}t^{(2)}_{j''}$ is at least $\frac14  \left(
n^{0.1} \right)^2$. Among these pairs the number of edges claimed by Breaker is upper bounded by $2\cdot n^{0.1}\cdot 10c_3 + 3n^{0.1}$, due to the bound on the edge degrees and since each step in Stage II lasts at most $3n^{0.1}$ rounds. Hence, the probability that for all $n^{0.1}$ choices of $j', j''$ the step of Stage~II is not complete, resulting in forfeiture of Maker, is at most $\left( \frac34 + o(1) \right)^{n^{0.1}} + \exp(- \Theta(n^{0.1}))$. 

If we apply the union bound for all $O(n^{0.7})$ steps of Stage~II, the probability that Maker completes the whole stage, and thus the whole game, without forfeiting, is still $1-o(1)$. As Stage II lasts $O(n^{0.8})$ rounds, the last whole phase lasts less than 
$\frac{\varepsilon n}{2a} + n^{0.9}$ rounds.
This proves the claim.
\end{proof}

Finally, note that $\operatorname{Pr}(\mathcal{B}_{1/\varepsilon}|\, \mathcal{E}_{1/\varepsilon-1} \land \mathcal{A}_{1/\varepsilon}) = 1$, since
under the assumption of $\mathcal{A}_{1/\varepsilon}$ all vertices belong to $V_M$ at the end of the last phase. Hence, putting everything together similarly to the proof of Theorem~\ref{thm:maker_degree}~\ref{m_deg_large_b}, we conclude that Maker wins with probability at least
$$
(1-o(1))\eps^{\sum_j |D_{j-1}|} 
\geq \frac{1}{(10b/a)^{11\cdot b/a}}
$$
where we use that the number of vertices with Breaker's degree at least $\frac{n}{10}$ cannot be larger than $\frac{2 B^{\text{(tot)}}}{0.1n} < \frac{11b}{a}$.
\end{proof}

\bigskip


\subsection{Perfect matching game with $b\leq 2a$}

\begin{proof}[Proof of Theorem~\ref{thm:maker_matching}~\ref{match_small_b}]
Let Maker play according to the same strategy
as for Theorem~\ref{thm:maker_matching}~\ref{match_large_b}, but with the following difference:
in Stage I Maker only forfeits the game if
more than $0.5n^{0.7}$ edge choices happen by executions of (2) when running 
over all failures made in (1).
Note that this in particular ensures that Maker
makes less than $\frac{n}{2} - 0.5n^{0.7}$
edge choices during Stage I.
Additionally this means that we ignore the splitting
of Maker's strategy into phases and the forfeiting condition regarding the length of such a phase.

\smallskip

\textbf{Strategy discussion:} We first prove the following.

\begin{claim}
	In Stage I, 
	Maker a.a.s. needs at most $n^{0.6}\log^3 n$
	edge choices executed according to (2)
	and, in particular, she does not forfeit
	in this stage.
\end{claim}

\begin{proof}
Throughout Stage I we have $|V_0|\geq n^{0.7}$
and $e(B)\leq \frac{b}{a}\left(\frac{n}{2} - 0.5n^{0.7}\right) \leq n - n^{0.7}$.
Hence, whenever Maker chooses a pair $x_1,x_2$ according to (1), the probability
that it is already claimed by PhantomBreaker is bounded
by $O(n/|V_0|^2)=O(n^{-0.4})$. 
Therefore, in expectation there are at most $O(n^{0.6})$
failures happening in (1), and
hence, applying Markov's inequality,
a.a.s.~at most $n^{0.6}\log n$ executions of (1) happen to be a failure.
We condition on this from now on.
Considering all of these executions of (1) with a failure, we aim to show that a.a.s.~at most $O(n^{0.6}\log^2 n)$ edge choices happen because of (2).

To do so, we distinguish the vertices according to their Breaker degree. We call a vertex $x_i$ 
\textit{bad} if $d_B(x_i)\geq 0.1n$ holds
at the moment when $x_i$ is chosen in (1), 
and we note that there can be only
$O(1)$ bad vertices throughout the whole game;
otherwise we call $x_i$ \textit{good}.

If a vertex $x_i$ is good at the moment of choice,
we then have $d_F(x_i,V_0)>0.4n$ or $d_F(x_i,V_M)>0.4n$
in each of the following $50 \log n$ rounds (we will see that is unlikely to need more rounds for the execution of (2)).
If $d_F(x_i,V_0)\geq 0.4n$, then in each repetition within (2)
we have with probability at least $0.4$ that $y\in V_0$ and $x_iy$ is free. Hence, 
applying Lemma~\ref{lemma:Chernoff_modified}(a)
with $p=0.4$ and $\delta=0.9$,
we see that with probability $1-O(n^{-2})$ after at most $20\log n $ repetitions the execution of (2) for $x_i$ is done.
If, otherwise, $d_F(x_i,V_M)\geq 0.4n$,
then first note that $|V_M|\geq 0.4 n$ (and $e(B)<n$),
and therefore at most a 0.1 fraction of the vertices in $V_M$
can have Breaker degree larger than $50$. Hence,
when one of the repetitions within (2) is executed,
with probability larger than $0.3$ we have that
Maker chooses a vertex $y$ such that
$x_iy$ is free and the current matching partner $y'$ of $y$
satisfies $d_B(y')\leq 50$, so that afterwards
the probability of having $y'z$ being free is $1-o(1)$.
Hence, each repetition succeeds with probability larger than $0.2$,
and we can conclude similarly that with probability $1-O(n^{-2})$ 
the execution of (2) for $x_i$ will be finished after at most $40 \log n$ edge choices.

If a vertex $x_i$ is bad at the moment of choice,
we still have
$d_F(x_i,V_0)>0.4n^{0.7}$ or $d_F(x_i,V_M)>0.4n^{0.7}$,
since $e(B)\leq n-n^{0.7}$.
If $d_F(x_i,V_0)\geq 0.4n^{0.7}$, then in each repetition within (2)
we have with probability $\Omega(n^{-0.3})$ that $y\in V_0$ and $x_iy$ is free. Hence, 
again applying Lemma~\ref{lemma:Chernoff_modified}(a),
we see that with probability $1-O(n^{-2})$ after at most $n^{0.3}\log^2 n$ repetitions, the execution of (2) for $x_i$ is done.
If otherwise $d_F(x_i,V_M)\geq 0.4n^{0.7}$
then, since $e(B)<n$, we see that
only $o(n^{0.7})$ vertices in $V_M$
can have Breaker degree larger than $n^{0.4}$. Hence,
when one of the repetitions within (2) is done,
with probability larger than $(0.4 - o(1))n^{-0.3}$
Maker chooses a vertex $y\in V_M$ such that
$x_iy$ is free and the current matching partner $y'$ of $y$
satisfies $d_B(y')\leq n^{0.4}$. In particular, for such $y'$ 
we afterwards have that
the probability of $y'z$ being free is $1-o(1)$.
Hence, each repetition succeeds with probability at least $0.3n^{-0.3}$,
and we can conclude similarly that with probability $1-O(n^{-2})$ 
the execution of (2) for $x_i$ will be finished after at most $n^{0.3}\log^2 n$ edge choices.

Now, doing a union bound over all $O(n^{0.6}\log n )$ steps with a failure in (1), we see that a.a.s.~whenever $x_i$ is good we need an extra of $O(\log n)$ edge choices, while if $x_i$ is bad we need at most $O(n^{0.3}\log^2 n)$ additional edge choices. Since the number of
steps involving bad vertices is bounded by $O(1)$, we see that
Maker succeeds with Stage I and makes at most $O(n^{0.6}\log^2 n)$
edge choices according to (2).
\end{proof}

Before proceeding to Stage~II, we also prove the following claim on Breaker's degrees.

\begin{claim}
At the end of Stage I, a.a.s.~every vertex $v\in V_0$ satisfies
$d_B(v)<n^{0.9}$.
\end{claim}

\begin{proof}
Note that at the moment when $|V_0|=n^{0.85}$ holds,
there can be at most $4n^{0.1}$ vertices $v_1,\ldots,v_t\in V_0$
with degree at least $0.5n^{0.9}$, since $e(B)<n$.
All other vertices do not have the chance of reaching
degree $n^{0.9}$ until the end of Stage I,
since this stage lasts only $O(n^{0.85})$ further rounds from this moment on.
Moreover, in the remaining part of Stage I we have that
all but at most
$n^{0.85}-2n^{0.7}-O(n^{0.6}\log^2 n)$ vertices of $V_0$ will be moved
to $V_M$ by rule (1). Therefore, the probability that at least one of the vertices $v_1,\ldots,v_t$ is not moved to $V_M$ until the end of Stage~I is bounded from above by
$4n^{0.1} \cdot \frac{O(n^{0.7})}{n^{0.85}} = o(1)$.
Thus, a.a.s.~$v_1,\ldots,v_t\notin V_0$ at the end of Stage I, while no other vertex in $V_0$ can reach Breaker degree of at least $n^{0.9}$.
\end{proof}

It remains to prove that Maker does a.a.s.~not forfeit Stage II. Using the degree bound from the above claim and the fact that by the strategy description Stage II lasts $O(n^{0.8})$ rounds, we see that whenever a pair $x_1,x_2$ is chosen, Breaker's degree of $x_1$ and $x_2$ is much smaller than the bound $\frac{n}{4}$ that we used in the analysis of Stage II in the proof of Claim~\ref{c:prob-last-ai-pm}; see the second part of the last phase therein. Now, following the same analysis word-by-word,
we see that Maker succeeds with probability $1-o(1)$.
\end{proof}

For our next proof regarding the connectivity game,
let us keep the following information.

\begin{remark}\label{remark:pm}
In the above proof
Maker a.a.s.~wins the $(a:b)$ perfect matching game,
with $b\leq 2a$,
against PhantomBreaker within $\frac{n}{2a} + O(n^{0.8})$
rounds.
\end{remark}

\bigskip


\subsection{Connectivity game with $b > 2a$}

\begin{proof}[Proof of Theorem~\ref{thm:maker_connectivity}~\ref{conn_large_b}]
Let $\varepsilon = \frac{a}{8b}$.
We consider the following randomized strategy for Maker:

Maker plays iteratively the same sequence of moves, as described in (1) and (2) below. If she does not manage to execute this sequence $n-1$ times within $1.1\cdot \frac{n}{a}$ rounds, she forfeits the game. 
Moreover, she splits her strategy into $\frac{1}{\varepsilon}$ stages $\mathcal{S}_i$, with $1\leq i\leq \frac{1}{\varepsilon}$, as follows:

\smallskip

\textbf{Stage $\mathcal{S}_i$:} 
Maker repeats the following sequence exactly
$\varepsilon(n-1)$ times, where again we avoid the use of rounding signs as they are not crucial for our calculations:

\begin{itemize}
\item[(1)] Maker chooses a vertex $v$ uniformly at random among all vertices of $K_n$ which were not chosen this way during 
stage $\mathcal{S}_i$. 
Let $C_v$ be the component in Maker's current graph $M$ which contains $v$.
\item[(2)] For up to $n^{0.2}$ edge choices, take an edge from
$E(C_v,V\setminus C_v)$ u.a.r.~from all such edges, but only until hitting an edge which is free.
If a free edge $e$ appears,
Maker claims it and proceeds to the next repetition of (1) and (2). Otherwise, i.e.~if no edge choice is successful, she forfeits the game.
\end{itemize}

\medskip

\textbf{Strategy discussion:}
We first note that Maker never claims an edge which closes a cycle.
Therefore, if the above sequence is executed $n-1$ times without forfeiting, then Maker wins the connectivity game. In the following, we aim to show 
that this indeed happens
with probability at least $\frac{1}{(10\cdot b/a)^{4\cdot b/a}}$.

\medskip

In order to analyse the strategy, we first introduce some notation and a useful colouring on the set of vertices. Assume that the game is in progress.
\begin{itemize}
\item We call a vertex $v$ \textit{bad} if $d_B(v)\geq \frac{n}{3}$.
\item When a vertex turns bad, we immediately give it the colour \textit{red}. When Maker in (1) chooses a vertex $v$ which is red, it gets coloured \textit{blue} immediately. 
\item We call a component $C$ in Maker's graph 
\textit{annoying} if
$\sum_{v\in C} d_B(v) \geq |C|\cdot n^{0.4}$.
\item We call a component $C$ in Maker's graph \textit{dangerous} if
$\sum_{v\in C} d_B(v) \geq (|C|-0.3) n$.
\end{itemize}

Moreover, note that during the whole process, since at most $1.1\cdot \frac{n}{a}$ rounds are played, Breaker claims at most $1.1\cdot \frac{bn}{a}$ edges. Therefore, 
only less than $\frac{4b}{a}$ vertices can become bad/red. 
Moreover, components of size at least $\frac{5b}{a}$ cannot become dangerous.
We next prove a useful claim that helps to bound the length of each stage.

\begin{claim}\label{clm:connectivity_number.rounds}
For $1 \leq i\leq \frac{1}{\varepsilon}$
the following holds:
$$
\Prob\left(\text{Maker does not forfeit in Stage $\mathcal{S}_i$, and $\mathcal{S}_i$ lasts more than $1.01\cdot \frac{\varepsilon n}{a}$ rounds} \right) = o(1).
$$
\end{claim}

\begin{proof}
We assume that Maker does not forfeit in the first $1.01\cdot \frac{\varepsilon n}{a}$ rounds of Stage 
$\mathcal{S}_i$. We will show that then with probability $1-o(1)$ this stage is already completed successfully.

For the moment, call a sequence in $\mathcal{S}_i$
a \textit{typical} sequence if in (1) the chosen vertex $v$ does not 
belong to an annoying component.
Within the first $1.01\cdot \frac{\varepsilon n}{a}$ rounds in Stage~$\mathcal{S}_i$ there can appear at most $O(n^{0.7})$ non-typical sequences by the following reason:
In each non-typical sequence we enlarge an annoying component by one vertex. If we would have more than 
$n^{0.7}$ non-typical sequences,
then by the definition of annoying components, there would need to be $\Omega(n^{0.7} \cdot n^{0.4}) = \omega(n)$ Breaker edges, in contradiction to the game lasting $O(n)$ rounds.

Now, these non-typical sequences together can lead to only $O(n^{0.9})=o(n)$ edge choices, as for each sequence the number of edge choices according to (2) is bounded by $n^{0.2}$. By the same reason, the last $n^{0.7}$ sequences in $\mathcal{S}_i$ can contribute only $o(n)$ edge choices.

Next, consider the remaining, i.e.~typical, sequences
of Stage~$\mathcal{S}_i$,
and let $X$ be the random variable
counting the typical sequences, except from the last $n^{0.7}$ sequences, which require more than one edge choice. 
We first aim to show that with high probability $X$
does not get too large.
To do so, consider the $j$-th typical sequence in Stage $\mathcal{S}_i$, call it $Seq_j$, and denote with $v_j$ the vertex chosen in step (1), and let $C_{v_j}$ be its component at the moment of choice.
Note that $|C_{v_j}|\leq n - n^{0.7}$, as we ignore the last
$n^{0.7}$ sequences.
Therefore and since $C_{v_j}$ is not annoying, we get
\begin{align*}
&
\Prob\left(\text{Maker takes more than one edge
in step (2) of sequence $Seq_j$}\right)\\
\leq ~
&
\Prob\left(\text{the first edge taken by Maker in (2) of sequence $Seq_j$ belongs to Breaker}\right)\\
\leq ~
&
\frac{|C_{v_j}|\cdot n^{0.4}}{|C_{v_j}|\cdot (n-|C_{v_j}|)}
\leq
\frac{|C_{v_j}|\cdot n^{0.4}}{|C_{v_j}|\cdot n^{0.7}} = n^{-0.3}.
\end{align*}
Therefore, the expected value of $X$ is 
$O(n^{0.7})$, and hence,
by applying Markov's inequality, we get
$\Prob\left(X > n^{0.75}\right) = o(1)$.

From now on we condition on the likely event $X \leq n^{0.75}$. Then we can bound the number of rounds in Stage $\mathcal{S}_i$ as follows:
The non-typical sequences and the last $n^{0.7}$ sequences contribute $o(n)$ edge choices. 
All except $X$ of the remaining sequences require only one edge choice each, leading to a total of at most 
$\varepsilon n$ edge choices.
The remaining $X$ typical sequences require at most $X\cdot n^{0.2}=o(n)$ edge choices in total. Thus, taking a sum over all sequences in $\mathcal{S}_i$, 
we end up with at most 
$(1+o(1))\varepsilon n$ edge choices and hence
at most $1.01\cdot \frac{\varepsilon n}{a}$
rounds. 
\end{proof}

Before coming to our next claims, let us first introduce some useful notation.
We let $\mathcal{S}_0$ be an empty stage
(with no rounds) before the game starts.
For $0\leq i\leq \frac{1}{\varepsilon}$ we define the following events: 
\begin{itemize}
\item $\mathcal{A}_{i,1}$: the event that Maker does not forfeit during Stage 
$\mathcal{S}_{i}$ according to step (2),
\item $\mathcal{A}_{i,2}$: the event that Stage $\mathcal{S}_{i}$ lasts at most $1.01\cdot \frac{\varepsilon n}{a}$ rounds,
\item $\mathcal{A}_{i,3}$: the event that no dangerous component appears during Stage $\mathcal{S}_{i}$, and
\item $\mathcal{A}_{i,4}$: the event that no vertex which was coloured red during Stage~$\mathcal{S}_{i-1}$ stays red until the end of Stage~$\mathcal{S}_{i}$.
\end{itemize}
Moreover, we set
$\mathcal{E}_i = \bigwedge_{j \leq i} (\mathcal{A}_{j,1} \land \mathcal{A}_{j,2} \land \mathcal{A}_{j,3} \land \mathcal{A}_{j,4})$.
Note that 
$\Prob(\mathcal{E}_0)=1$ holds trivially,
and that Maker wins if 
$\mathcal{E}_{1/\varepsilon}$ is satisfied,
as the events $\mathcal{A}_{j,2}$ ensure that the game lasts less than $1.1\cdot \frac{n}{a}$ rounds in total,
and the events $\mathcal{A}_{j,1}$ ensure that Maker never forfeits in any of the $n-1$ sequences of her strategy.
Hence, it remains to show that
$\Prob(\mathcal{E}_{1/\varepsilon})\geq \frac{1}{(10\cdot b/a)^{4\cdot b/a}}$.
In order to do so, we first prove a couple of claims.

\begin{claim}\label{claim:Ai1}
For every $i\leq \frac{1}{\varepsilon}$ we have
$\Prob(\overline{\mathcal{A}_{i,1}}|\mathcal{E}_{i-1})
=o(1)$.
\end{claim}

\begin{proof}
Assume that $\mathcal{E}_{i-1}$ holds.
By Claim~\ref{clm:connectivity_number.rounds}
we know that
with probability $1-o(1)$,
if Maker forfeits in Stage~$\mathcal{S}_{i}$,
then this must happen during the first 
$1.01\cdot \frac{\varepsilon n}{a}$ rounds of this stage. 

Now, let $Seq_j$ be any fixed sequence within the first 
$1.01\cdot \frac{\varepsilon n}{a}$ rounds
of Stage~$\mathcal{S}_{i}$. 
In order to succeed with a union bound over all
sequences in this stage, it is enough to show
that the probability of forfeiting during $Seq_j$ because of step (2) is upper bounded by $\exp(-n^{-0.1})$. Assume that $v$ is the vertex chosen in step (1) of sequence $Seq_j$, and let $C_v$ be its component. 

We first consider the case that $|C_v|\leq \frac{5b}{a}$. Then
let $C_1,\ldots,C_{\ell}$ be the components of $M[C_v]$
at the end of Stage $\mathcal{S}_{i-1}$.
Under assumption of $\mathcal{E}_{i-1}$
we have that $\bigwedge_{j\leq i-1} \mathcal{A}_{j,3}$
holds. Therefore, at the end of Stage $\mathcal{S}_{i-1}$ no such component
is dangerous, giving
$$
\sum_{u\in C_v} d_B(u)
=
\sum_{k=1}^{\ell} \sum_{u\in C_k} d_B(u)
<
\sum_{k=1}^{\ell} (|C_k|n - 0.3n)
\leq
|C_v|n - 0.3n.
$$ 
Since we now consider only the first $1.01\cdot \frac{\varepsilon n}{a}$ rounds of $\mathcal{S}_{i}$,
in the meantime Breaker can play at most
$1.01\cdot \frac{\varepsilon nb}{a}$ further edges, and hence increase the above sum of degrees by at most
$1.01\cdot \frac{\varepsilon nb}{a} + \binom{|C_v|}{2}<0.15n$, leading to
$
\sum_{u\in C_v} d_B(u)
<
|C_v|n - 0.15n
$
throughout sequence $Seq_j$.
Hence, more than $0.1n$ edges in $E(C_v,V\setminus C_v)$ are free throughout this sequence, and it follows
that the probability of forfeiting in step (2) is bounded from above by
\begin{align*}
\left( 1 - \frac{0.1n}{|C_v|\cdot (n-|C_v|)} \right)^{n^{0.2}}
& \leq 
\left( 1 - \frac{0.1n}{5bn/a} \right)^{n^{0.2}} 
=
\left( 1 - \frac{a}{50b} \right)^{n^{0.2}} 
<
\exp\left( - n^{0.1}\right)
\end{align*}
for large enough $n$, since $a$ and $b$ are constants.

If instead $|C_v|\geq n - \frac{5b}{a}$,
then we can calculate similarly,
as each component in $V\setminus C_v$ was not dangerous
at the end of Stage $\mathcal{S}_i$,
again leading to more than $0.1n$ free edges in $E(C_v,V\setminus C_v)$.

Hence, it remains to consider the case 
$\frac{5b}{a} < |C_v| < n-\frac{5b}{a}$.
Then $e(C_v,V\setminus C_v) > \frac{4bn}{a}$
for large enough $n$,
while we also have that 
$e_B(C_v,V\setminus C_v) < \frac{2bn}{a}$ during the considered sequence, since the game lasts at most $1.1\cdot \frac{n}{a}$ rounds in total. Hence, in this case
the probability of forfeiting in step (2) is bounded from above by
$0.5^{n^{0.2}} <
\exp\left( - n^{0.1}\right)$
for large enough $n$.
\end{proof}

\begin{claim}
For every $i\leq \frac{1}{\varepsilon}$ we have
$\Prob(\overline{\mathcal{A}_{i,2}}|\mathcal{E}_{i-1},\mathcal{A}_{i,1}) = o(1)$.
\end{claim}

\begin{proof}
Using Claim~\ref{clm:connectivity_number.rounds}
and Claim~\ref{claim:Ai1}, we immediately see that
\begin{align*}
\Prob(\overline{\mathcal{A}_{i,2}}|\mathcal{E}_{i-1},\mathcal{A}_{i,1})
=
\frac{\Prob(\overline{\mathcal{A}_{i,2}} \land \mathcal{A}_{i,1}|\mathcal{E}_{i-1})}{
1- \Prob(\overline{\mathcal{A}_{i,1}}|\mathcal{E}_{i-1}) }
=
\frac{o(1)}{1-o(1)} 
=
o(1)
\end{align*}
which proves the claim.
\end{proof}

\begin{claim}
For every $i\leq \frac{1}{\varepsilon}$ we have
$\Prob(\overline{\mathcal{A}_{i,3}}|\mathcal{E}_{i-1},\mathcal{A}_{i,1},\mathcal{A}_{i,2})=0$.
\end{claim}

\begin{proof}
Assume that all of 
$\mathcal{E}_{i-1},\mathcal{A}_{i,1},\mathcal{A}_{i,2}$ hold. We aim to show that this implies $\mathcal{A}_{i,3}$.
For contradiction,
assume that a dangerous component appears during 
$\mathcal{S}_{i}$, call it $C$ and denote its size by $c:=|C|$. Recall that $c\leq \frac{5b}{a}=O(1)$.
At the moment of becoming dangerous, we have
$\sum_{v\in C} d_B(v) \geq (|C|-0.3) n$.
Since every stage so far lasted at most $1.01\cdot \frac{\varepsilon n}{a}$ rounds,
by conditioning on $\mathcal{E}_{i-1}$ and $\mathcal{A}_{i,2}$, Breaker so far claimed less than 
$0.3n$ edges in every two consecutive stages, and
so at some point in Stage $\mathcal{S}_{i-2}$ we must have had
$\sum_{v\in C} d_B(v) \geq (|C|-0.6) n$.
But then every vertex $v\in C$ must 
satisfy $d_B(v)\geq \frac{n}{3}$ 
already in Stage $\mathcal{S}_{i-2}$,
and hence was coloured red
in or before $\mathcal{S}_{i-2}$.
However, by assuming $\mathcal{E}_{i-1}$,
we have that $\bigwedge_{j\leq i-1} \mathcal{A}_{j,4}$
holds, and hence each such vertex $v$ was coloured blue until the end of Stage $\mathcal{S}_{i-1}$. But this means that until the end of Stage $\mathcal{S}_{i-1}$ the vertex $v$ was chosen in some step (1),
and since Maker did not forfeit until the end of Stage
$\mathcal{S}_{i-1}$,
at that point then the component $C_v\subset C$ containing $v$ was extended by an edge.
Since this must have happened for each of the $c$ vertices in $C$, always when they were coloured from red to blue,
we can conclude that $C$ must have obtained at least $c$ edges this way, in contradiction to $C$ having $c$ vertices and no cycles.
\end{proof}

\begin{claim}
For every $i\leq \frac{1}{\varepsilon}$ we have
$\Prob(\overline{\mathcal{A}_{i,4}}|\mathcal{E}_{i-1},\mathcal{A}_{i,1},\mathcal{A}_{i,2},\mathcal{A}_{i,3})\leq 1- (1-o(1))\varepsilon^{t_{i-1}}$,
where $t_{i-1}$ is the number of vertices that turned red during Stage $\mathcal{S}_{i-1}$.
\end{claim}

\begin{proof}
Assume that $\mathcal{E}_{i-1},\mathcal{A}_{i,1},\mathcal{A}_{i,2},\mathcal{A}_{i,3}$ hold.
Let $R$ be the set of the $t_{i-1}$ vertices that turned red in Stage $\mathcal{S}_{i-1}$,
and note that $t_{i-1}=O(1)$, because of our general bound on the number of bad vertices. All the vertices chosen according to (1) in Stage $\mathcal{S}_{i}$ form a set
of size $\varepsilon(n-1)$ chosen
uniformly at random from all sets of this size.
By the definition of the vertex colouring we get 
\begin{align*}
& \Prob\left( \text{all vertices of $R$ are blue at the end of Stage $\mathcal{S}_{i}$} \right) \\
\geq ~
& \Prob\left( \text{all vertices of $R$ are chosen in step (1) of a sequence of $\mathcal{S}_{i}$} \right) \\
= ~
& \frac{\binom{n-t_{i-1}}{\varepsilon(n-1) - t_{i-1}}}{\binom{n}{\varepsilon(n-1)}}
=
\prod_{i=1}^{t_{i-1}} 
\frac{\varepsilon(n-1) - i + 1}{n - i+ 1}
=
\varepsilon^{t_{i-1}} (1-o(1))
\end{align*}
and hence the claim follows.
\end{proof}

Finally, putting all the above claims together 
we conclude that for every $i\leq \frac{1}{\varepsilon}$,
\begin{align*}
\Prob\left(\overline{\mathcal{E}_{i}}|\mathcal{E}_{i-1}\right)
& =
\Prob\left(\overline{\mathcal{A}_{i,1}}\lor \overline{\mathcal{A}_{i,2}}\lor \overline{\mathcal{A}_{i,3}}\lor \overline{\mathcal{A}_{i,4}}|\mathcal{E}_{i-1}\right) \\
& =
\sum_{j\in [4]}
\Prob\left(\overline{\mathcal{A}_{i,j}} \land \left(\bigwedge_{\ell < j} \mathcal{A}_{i,\ell} \right) \Big| \mathcal{E}_{i-1}\right) 
\leq 
\sum_{j\in [4]}
\Prob\left(\overline{\mathcal{A}_{i,j}} \Big| \mathcal{E}_{i-1} \land \left(\bigwedge_{\ell < j} \mathcal{A}_{i,\ell} \right) \right) \\
& \leq
1 + o(1) - (1-o(1))\varepsilon^{t_{i-1}} 
\end{align*}
and therefore
$\Prob\left(\mathcal{E}_{i}|\mathcal{E}_{i-1}\right)
\geq 
(0.99\varepsilon)^{t_{i-1}}$.
In particular, by recalling that Breaker cannot create more than $\frac{4b}{a}$ bad vertices,
which implies $\sum_j t_j \leq \frac{4b}{a}$, and by using that
$\mathcal{E}_{\frac{1}{\varepsilon}}\subset \mathcal{E}_{\frac{1}{\varepsilon}-1} \subset \ldots \subset \mathcal{E}_{0}$ and $\Prob(\mathcal{E}_0)=1$, we get
\begin{align*}
\Prob\left(\text{Maker wins}\right)
\geq 
\Prob\left(\mathcal{E}_{\frac{1}{\varepsilon}}\right)
=
\prod_{i=1}^{1/\varepsilon} \Prob\left(\mathcal{E}_{i}|\mathcal{E}_{i-1}\right) 
\geq
\prod_{i=1}^{1/\varepsilon}
(0.99\varepsilon)^{t_{i-1}}
>
\left(\frac{1}{10 b/a}\right)^{4b/a}
\end{align*}
which completes the proof of the theorem.
\end{proof}

\bigskip


\subsection{Connectivity game with $b\leq 2a$}

\begin{proof}[Proof of Theorem~\ref{thm:maker_connectivity}~\ref{conn_small_b}]
Maker plays according to the following strategy. 

\smallskip

\textbf{Stage I:} Maker claims a perfect matching
according to the strategy given in the proof of
Theorem~\ref{thm:maker_matching}~\ref{match_small_b}.
If this stage lasts more than $\frac{n}{2a}+n^{0.9}$ rounds, then she forfeits the game.

\smallskip

\textbf{Stage II:} As long as there are components of size 2 in Maker's graph, she does the following iteratively.
\begin{itemize}
\item[(i)] 
If there are less than $n^{2/3}$ components of size 2,
she proceeds to step (ii). Otherwise she picks two components $C_1,C_2$ of size 2 u.a.r., and then chooses an edge $f\in E(C_1,C_2)$
u.a.r. 
If this edge $f$ is free, Maker claims it
and repeats the iteration for her next edge. 
If the edge $f$ is not free, this choice is called a \emph{failure}.
Then for each component $C\in \{C_1,C_2\}$ Maker performs the step (iii) before starting the iteration again.
\item[(ii)] 
There are less than $n^{2/3}$ components of size 2. 
Maker picks one component $C$ of size 2 u.a.r., and then chooses an edge $f\in E(C,V\setminus C)$ u.a.r. 
If this edge $f$ is free, she claims it
and repeats the iteration for her next edge. 
If the edge $f$ is not free, this choice is a failure.
Then for the component $C$ she does the step (iii)
before starting the iteration again.
\item[(iii)] Let $C$ be a component as described in (i) or (ii) where a failure happened. Then Maker chooses the edges from $E(C,V\setminus C)$ u.a.r.~until having the first success. When a success occurs, she claims the corresponding edge. 
\end{itemize}

Maker forfeits the game if this stage lasts more than $1.1\cdot \frac{n}{4a}$ rounds.

\smallskip

\textbf{Stage III:} As long as Maker's graph is not connected, she does the following iteratively.

She picks a component $C$ u.a.r. and then chooses the edges from $E(C,V\setminus C)$ 
u.a.r. until she has a success. When a success happens,  she claims the corresponding edge. 

Maker forfeits the game, if this stage lasts more than $1.1\cdot \frac{n}{4a}$ rounds before her graph is connected.

\bigskip

\noindent
\textbf{Strategy discussion:}
Maker a.a.s. does not forfeit in Stage I
due to Remark~\ref{remark:pm}. Hence, we
can start with analysing Stage II.
We start with the following claim.

\begin{claim}\label{clm:Connectivity:failuresII}
The following holds a.a.s.:
each time when Maker executes (iii) for some component $C$,
she needs to choose at most $\log^2 n$ edges until
a success happens.
\end{claim}

\begin{proof}
By the bound on the number of rounds of Stages I and II,
Breaker claims at most 
$b\cdot \left(\frac{n}{2a} + n^{0.9} + 1.1\cdot \frac{n}{4a} \right) < 1.6 n$ 
edges until the end of Stage II. Since $e(C,V\setminus C)\geq 2(n-2)$ for every component $C$ picked in Stage II, it follows that each time Maker chooses an edge in (iii)
the probability of a failure is bounded from above by
$\frac{1.6n}{2(n-2)} < 0.9$, and hence the probability of success
is at least $0.1$. The claim now follows by applying Lemma~\ref{lemma:Chernoff_modified}(a) with $p=0.1$,
and a union bound over all $O(n)$ iterations executed in Stage II.
\end{proof}

\begin{claim}
The following holds a.a.s.: Maker does not forfeit during Stage II and this stage lasts 
$\frac{n}{4a} + O(n^{2/3}\log^2 n)$
rounds.
\end{claim}

\begin{proof}
We condition on Claim~\ref{clm:Connectivity:failuresII}
and first aim to bound the number of iterations
in Stage II where a failure happens in (i) or (ii).
If Maker plays according to (i),
then Maker has $\Omega(n^{4/3})$ options to choose $C_1,C_2$.
Since $e(B)=O(n)$ it follows that in every execution of (i) a failure can only happen with probability $O(n^{-1/3})$.
Hence, applying Lemma~\ref{lemma:Chernoff_modified}(b), it follows that
a.a.s.~only $O(n^{2/3})$ executions of (i) will have a failure.
Moreover, as (ii) cannot be executed more than $n^{2/3}$
times, the number of failures caused by (ii) is also bounded by $O(n^{2/3})$. 

Now, by Claim~\ref{clm:Connectivity:failuresII}
we know that, whenever a failure happens in (i) or (ii), Maker needs only $O(\log^2 n)$ additional edge choices in (iii). Therefore, a.a.s.~the total number of
failures is $O(n^{2/3}\log^2 n)$. Hence, Stage II a.a.s.~lasts $\frac{n}{4a} + O(n^{2/3}\log^2 n)$
rounds and Maker does not forfeit. 
\end{proof}

\smallskip

Now, conditioning on the above claims, it remains to show that Maker a.a.s. does not forfeit during Stage III, finally leading to a spanning connected graph. 
We start with the following variant of Claim~\ref{clm:Connectivity:failuresII}.

\begin{claim}\label{clm:Connectivity:failuresIII}
The following holds a.a.s.:
each time when Maker picks a component $C$,
she needs to choose at most $\log^2 n$ edges until
a success happens.
\end{claim}

\begin{proof}
When Maker enters Stage III, each of her components has size at least 4. 
By the already proven bounds on the number of rounds
for Stages I and II, and the given bound for forfeiting in Stage III, we have that
Breaker claims at most 
$b\cdot \left(\frac{n}{2a} + \frac{n}{4a} + 1.1\cdot \frac{n}{4a} + o(n)\right) < 2.1n$ 
edges until the end of Stage III. Since $e(C,V\setminus C)\geq 4(n-4)$ for every component $C$ picked in Stage III, it follows that each time Maker chooses an edge in Stage III
the probability of a failure is bounded from above by
$\frac {2.1n}{4(n-4)} < 0.6$. The rest of the proof then is analogous to Claim~\ref{clm:Connectivity:failuresII}.
\end{proof}

We will condition on this from now on.
Moreover note that in Stage II at most $\frac{n}{4}$
iterations happen, since at the beginning of Stage III the number of components is bounded by $\frac{n}{4}$.
Next, we aim to show that a.a.s.~only $O(n^{0.9})$ iterations will have at least one failure. For this it is enough
to consider all but the last $n^{0.2}$ iterations, as these last iterations cannot contribute so much to this number.
We now call a component $C$ \textit{bad} if $e_B(C,V\setminus C)\geq n^{0.9}$, and note that
the number of bad components is bounded by $O(n^{0.1})$.
In particular, since we have at least $n^{0.2}$ components (by ignoring the last $n^{0.2}$ iterations), 
each time Maker picks a new component $C$ u.a.r.,
the probability of this component being bad is bounded by
$O(n^{0.1}/n^{0.2})=O(n^{-0.1})$.
By applying Lemma~\ref{lemma:Chernoff_modified}(b) we get that a.a.s.~in $O(n^{0.9})$
iterations Maker chooses a bad component.
In all the other iterations, i.e.~when $C$ is not bad,
the probability of having a failure when 
choosing the first edge in $E(C,V\setminus C)$
is bounded by $O(n^{0.9}/n)=O(n^{-0.1})$.
Hence, applying  Lemma~\ref{lemma:Chernoff_modified}(b) once more we get that a.a.s.~only $O(n^{0.9})$ of these iterations will have at least one failure.
Hence, putting everything together, we get that in total there can be at most $O(n^{0.9})$
iterations with at least one failure,
leading to a total of $O(n^{0.9}\log^2 n)$ failures
throughout Stage~III, because of Claim~\ref{clm:Connectivity:failuresIII}.
In particular, it follows that Stage~III a.a.s.~lasts at most $\frac{n}{4a} + O(n^{0.9}\log^2 n)$
rounds and Maker does not forfeit.
\end{proof}

\bigskip


\subsection{Hamiltonicity game with $b > a$}

\begin{proof}[Proof of Theorem~\ref{thm:maker_hamilton}~\ref{ham_large_b}]
Maker's strategy consists of two stages which are described below. 
At some points during the strategy, Maker might try to claim an edge which she already claimed earlier. In that case, she just claims the next edge according to her strategy. Moreover,
if she does not win the game within the first 
$\frac{n}{a} + n^{0.9}$ rounds, she forfeits the game.

\smallskip

\textbf{Stage I:} In this stage, Maker builds a Hamilton path. She starts the game with an arbitrary edge $uv$ forming a path $P$ on two vertices. Afterwards, she aims to extend $P$ in steps until it is Hamiltonian, where extending means to 
create a longer path using all current vertices of $P$ and at least one new vertex $x$ such that the updated path $P$ has one of the previous endpoints and the new vertex $x$ as its endpoints.

Throughout the strategy let $V_0=V\setminus V(P)$.
In order to do the described extension, in each step Maker chooses u.a.r. a vertex $x \in V_0$, and the endpoint $y$ of $P$ which was not added in the last step (in the first step she takes $y=u$). She then proceeds as follows:

\begin{enumerate}
\item[(1)] She first tries to claim $xy$. If successful, that completes the step, and $x$ becomes a new endpoint to $P$.

\item[(2)]
Otherwise, let $P = (v_1,v_2,\ldots,v_{v(P)})$ be enumerated in order starting at $v_1 = y$. 
Maker first aims to create large stars centered in $x$ and $y$:
\begin{enumerate}
\item[(i)] She first tries to claim 
$\frac{a}{20b}\log n\cdot \log\log n$ distinct edges
between $x$ and $V\setminus \{x,y\}$ u.a.r., 
stores their endpoints except from $x$ in a set $V_x$, and sets $V_x'\subset V_x$ to be those endpoints for which the edge choice was successful.
\item[(ii)] She then tries to claim 
$\frac{a}{20b}\log n\cdot \log\log n$ distinct edges
between $y$ and $V\setminus (\{x,y\}\cup V_x)$ u.a.r., 
stores their endpoints except from $y$ in a set $V_y$, and sets $V_y'\subset V_y$ to be those endpoints for which the edge choice was successful.
\end{enumerate}
Afterwards, she chooses u.a.r. a pair of vertices $(x',y') \in V_x' \times V_y'$. Depending on whether $x'$ and $y'$ are part of the path $P$, she tries to claim one of the following edges:
\begin{itemize}
    \item[(a)] If $x' \in V_0$ and $y' \in V_0$, she tries to claim the edge $x'y'$. In case of success, the sequence $(x,x',y',y,v_2,\ldots,v_{v(P)})$ forms a new path on $V(P) \cup \{x,x',y'\}$, completing the step.
    
    \item[(b)] If $x'=v_i \in V(P)$ and $y'=v_j \in V(P)$ and $i < j$, she tries to claim $v_{i-1}v_{j+1}$. In case of success, the sequence 
    $(x,x',v_{i+1},\ldots,v_{j-1},y',y,v_2,\ldots,v_{i-1},v_{j+1},v_{j+2},\ldots,v_{v(P)})$ forms a new path on $V(P) \cup \{x\}$, completing the step. Otherwise, if $i > j$, she tries to claim $v_{j-1}v_{i+1}$ and creates a new path analogously.
    
    \item[(c)] If $x'=v_i \in V(P)$ and $y' \in V_0$, she tries to claim the edge $v_{i+1}y'$. In case of success, the sequence $(x,x',v_{i-1},\ldots,v_2,y,y',v_{i+1},\ldots,v_{v(P)})$ forms a new path on $V(P) \cup \{x,y'\}$, completing the step.
    
    \item[(d)] If $x' \in V_0$ and $y'=v_i \in V(P)$, she tries to claim the edge $x'v_{i-1}$. In case of success, the sequence $(x,x',v_{i-1},\ldots,v_2,y,y',v_{i+1},\ldots,v_{v(P)})$ forms a new path on $V(P) \cup \{x,x'\}$, completing the step.
\end{itemize}

Maker repeats choosing pairs of $V_x' \times V_y'$ u.a.r. and tries to claim the corresponding edge until she succeeds once or fails for $2\log n$ consecutive attempts. 
In the former case, $P$ is updated according to the described sequence in (a)--(d). 
In the latter case, Maker forfeits the game.
\end{enumerate}
	
\textbf{Stage II:}
If Maker completed Stage I without forfeiting, she has created a Hamilton path $P = (v_1,\ldots,v_n)$. She then considers the path $P' = (v_1,\ldots,v_{n-1})$ and creates a new Hamilton path with endpoints $v_n$ and $v_{n-1}$ using the strategy from Stage~I. Thus, she either forms a Hamilton cycle, or forfeits the game.

\medskip

\textbf{Strategy discussion:}
If Maker never forfeits, she creates a Hamilton cycle
and wins the game. We aim to show that this indeed happens with probability at least
$\frac{1}{(10b/a)^{21b/a}}$.
We start our analysis by showing the following claim which states that Maker is very likely able to successfully execute (2) under some assumptions.

\begin{claim}\label{c:stage1(2)-ham}
    Assume that Maker tries to connect two vertices $x$ and $y$ in Stage~I, and both vertices have Breaker's degree less than $n - \frac{180bn}{a\log\log n}$, and the game lasted less than $(1+o(1))\frac{n}{a}$ rounds so far. Then, with probability at least $1 - O(\frac{1}{n})$, Maker successfully executes (2) and makes at most $\frac{a}{5b}\log n\cdot \log\log n$ edge choices in the process.
\end{claim}
\begin{proof}
    Maker makes at most 
    $2\frac{a}{20b}\log n\cdot \log\log n$ edge choices
    according to (i) and (ii), and at most 
    $2\log n$ edge choices when repeating (a)--(d).  
    This already proves the claimed bound on the number 
    of edge choices.

    We next argue that $|V_x'| \geq 4\log n$
    with probability $1-O(\frac{1}{n})$: 
    As Breaker's degree of $x$ is less than 
    $n - \frac{180bn}{a\log\log n}$ when Maker starts executing the strategy given by (2) and since Maker will make no more than $\frac{a}{5b}\log n\cdot \log\log n$ edge choices when executing (2), Breaker's degree of $x$ will not exceed 
$n - \frac{170bn}{a\log\log n}$ in the process.
Hence, each time Maker tries to claim an edge in (i)
the probability of success is larger than
$\frac{160b}{a\log\log n}$.
The expected number of successes in (i) is therefore at least $8\log n$, and by Lemma~\ref{lemma:Chernoff_modified}(a) with $p=\frac{160b}{a\log\log n}$, we get that
the probability for having $|V_x'| < 4\log n$ is bounded by $\frac{1}{n}$. 
By the same argument, we have $|V_y'| \geq 4\log n$
    with probability $1-O(\frac{1}{n})$.
We condition on $|V_x'| \geq 4\log n$ and $|V_y'| \geq 4\log n$ going forward.

Next, let $V^{\text{(bad)}} \subset V$ consist of all vertices from $V_0$ that have Breaker's degree more than $\log\log n$, and of all vertices from $V(P)$ that have a neighbour in $P$ with Breaker's degree more than $\log\log n$, immediately before the execution of (2) is started. 
We want to show that we have 
$|V^{\text{(bad)}} \cap V_x| \leq 1.5\log n$
with probability $1-O(\frac{1}{n})$.
To do so, first note 
$|V^{\text{(bad)}}|\leq (1+o(1))\frac{4bn}{a\log\log n}$, since
Breaker could create only $(1+o(1))\frac{2bn}{a\log\log n}$ vertices of degree more than $\log\log n$ in the at most $(1+o(1))\frac{n}{a}$ rounds, and each of those vertices can contribute at most $2$ to the value $|V^{\text{(bad)}}|$.
Hence, whenever Maker makes an edge choice $xv$ according to (i),
the probability that $v$ belongs to $V^{\text{(bad)}}\cap V_X$ is upper bounded by 
$(1+o(1))\frac{4b}{a\log\log n}$. 
As Maker has $\frac{a}{20b}\log n\cdot \log\log n$ 
edge choices in (i), the expected size of
$V^{\text{(bad)}}\cap V_X$ is at most $(\frac15 + o(1))\log n$.
Applying Lemma~\ref{lemma:Chernoff_modified}(c) with $p=(1+o(1))\frac{4b}{a\log\log n}$, we see that
$|V^{\text{(bad)}} \cap V_x| > 1.5\log n$ happens with probability smaller than $\frac{1}{n}$. We condition on $|V^{\text{(bad)}} \cap V_x| \leq 1.5\log n$ for the remainder of the proof. 

Let $E^{(2)}$ be the set of edges which are assigned to the pairs $V_x' \times V_y'$ according to the strategy, i.e.~which Maker should claim according to the cases (a)--(d) described in (2). Since this assignment is injective, we have $|E^{(2)}| = |V_x'|\cdot|V_y'| \geq 4\log n \cdot |V_y'|$. We want to show that $|E^{(2)} \cap E(B)| \leq 2\log n \cdot |V_y'|$ holds throughout the process. 
Before the start of the execution of (2), for each vertex $v\in V^{\text{(bad)}} \cap V_x$,  Breaker could only claim $|V_y'|$ edges of $E^{(2)}$ that are assigned to pairs $(v,y')$. Moreover, for every vertex $v \in V_x \setminus V^{\text{(bad)}}$ he could claim at most $2\log\log n$ edges of $E^{(2)}$ assigned to pairs $(v,y')$, where the factor $2$ comes from the fact that we might need to consider the edges adjacent to two vertices of degree at most $\log\log n$ if $v \in V(P)$. Furthermore, Breaker can claim at most 
$\frac15\log n\cdot \log\log n$ further edges of 
$E^{(2)}$ while Maker executes her strategy. 
In total, this gives 
\begin{align*}
|E^{(2)} \cap E(B)| &\leq 1.5\log n \cdot |V_y'| + 2\log\log n \cdot |V_x| + \frac15\log n\cdot \log\log n\\ 
& = (1.5 + o(1))\log n \cdot |V_y'| \leq 2\log n \cdot |V_y'|\end{align*}
throughout the process.
Thus, whenever Maker takes an edge of $E^{(2)}$, she succeeds with probability at least $\frac12$. The probability to fail for $2\log n$ consecutive attempts is therefore bounded by $\frac{1}{n}$.
\end{proof}

The remainder of the analysis will closely follow the analysis of Theorem~\ref{thm:maker_matching}~\ref{match_large_b}. Let $\eps = \frac{a}{10b}$, and split the game into 
$\frac{1}{\eps}$ phases, where the $i$-th phase finishes when $V_0$ contains less that $n - i\eps n$ vertices. We consider the following events:
\begin{itemize}
\item event $\mathcal{A}_i$: Maker completes the $i$-th phase without forfeiting, and the phase lasts at most $\frac{\eps n}{a} + \sqrt{n}\log^2 n \cdot \log\log n$ rounds.
\item event $\mathcal{B}_i$: 
let $D_{i-1}$ be the vertices
whose Breaker degree increased to $\frac{n}{10}$ during phase $i-1$, then all of $D_{i-1}$ are removed from $V_0$ until the end of the $i$-th phase.
\end{itemize}

Moreover, let $\mathcal{E}_j = \bigwedge_{i \leq j} (\mathcal{A}_i \land \mathcal{B}_i)$. Note that 
by the given forfeiting condition on the total 
number of rounds, the game will last for at most $\frac{n}{a}+n^{0.9}$ rounds. In that case the total number of Breaker's edges is upper bounded by $B^{\text{(tot)}}:=(1+o(1)) \frac{bn}{a}$ throughout the game.
If additionally all the events $\mathcal{E}_i$, $i \leq \frac{1}{\eps}-1$ occur, then Breaker's maximum degree in $V_0$ will never surpass $(1+o(1))(\frac{1}{10} + \frac{2b\eps}{a}) n < \frac{n}{3}$,
with the same argument as in the proof
of Theorem~\ref{thm:maker_matching}~\ref{match_large_b}.

We next prove a few claims similarly to the claims used for Theorem~\ref{thm:maker_degree}~\ref{m_deg_large_b}
and Theorem~\ref{thm:maker_matching}~\ref{match_large_b}.

\begin{claim} \label{c:prob-ai-ham}
For every $i<\frac{1}{\eps}$ we have $\operatorname{Pr}(\mathcal{A}_i|\, \mathcal{E}_{i-1}) = 1-o(1)$.
\end{claim}

\begin{proof}
We assume $\mathcal{E}_{i-1}$. Because 
$i < \frac{1}{\eps}$, we know that $|V_0| > \eps n$. Moreover, Breaker has claimed at most $B^{\text{(tot)}}$ edges so far. Then, when Maker tries to extend $P$ at an endpoint $y$, the probability of having a failure in (1) can be bounded by $\frac{d_B(y)}{|V_0|} \leq \frac{d_B(y)}{\eps n}$. Since every vertex $v$ is only used at most once as the endpoint $y$ as described in the strategy, we can bound the expected number of failures caused by (1) by $\sum_{v \in V} \frac{d_B(v)}{\eps n} \leq \frac{2 B^{\text{(tot)}}}{\eps n} = O(1)$. Hence,
by Markov's inequality it follows that the probability that more than $\log n$ such failures happen, is bounded by $o(1)$. 

From now on, we condition on at most $\log n$ failures happening in the $i$-th phase.
For each failure, Maker proceeds with (2). 
By Claim~\ref{c:stage1(2)-ham}, Maker successfully corrects the failure within at most $\frac{1}{5b}\log n\cdot \log\log n$ rounds with probability at least $1 - O(\frac{1}{n})$, because Breaker's maximum degree within $V_0$ and the endpoints of $P$ can be bounded by $(1+o(1))\frac{n}{3}$, by the explanation before this claim and because the endpoints of $P$ belonged to $V_0$ before Maker did the last two extensions of $P$. 
Taking a union bound over all of the at most $\log n$ failures, we see that a.a.s.~Maker does not forfeit in the $i$-th phase and the phase lasts at most $\frac{\eps n}{a} + \frac{1}{5b}\log^2 n \cdot \log\log n$ rounds.
\end{proof}

\begin{claim} \label{c:prob-bi-ham}
For every $i<\frac{1}{\eps}$ we have $\operatorname{Pr}(\mathcal{B}_i|\, \mathcal{E}_{i-1} \land \mathcal{A}_i) \geq (1-o(1)) \varepsilon^{|D_{i-1}|}$.
\end{claim}

\begin{proof}
The proof is analogous to the proofs of
Claim~\ref{c:prob-bj-min-k} and Claim~\ref{c:prob-bi-pm}.
\end{proof}

\begin{claim} \label{c:prob-last-ai-ham}
We have $\operatorname{Pr}(\mathcal{A}_{1/\eps}|\, \mathcal{E}_{1/\eps-1}) = 1-o(1)$.
\end{claim}

\begin{proof}
If Maker indeed finishes this last phase in at most 
$\frac{\eps n}{a} + \sqrt{n}\log^2 n \cdot \log\log n$ rounds (which we will show is going to happen a.a.s.) and because we condition on $\mathcal{B}_{1/\eps - 1}$, we can bound Breaker's degree for the whole phase by $(1+o(1))(\frac{1}{10} + \frac{2b\eps}{a}) n < \frac{n}{3}$. We split the analysis of the last phase depending on whether $|V_0| \geq \sqrt{n}$. While $|V_0| \geq \sqrt{n}$, we can use the same argument as for Claim~\ref{c:prob-ai-ham} to show that the expected number of failures using (1) is bounded by $\frac{2B^{(tot)}}{\sqrt{n}} = O(\sqrt{n})$. Using Markov's inequality, we know that the probability of having less than $\sqrt{n}\log n$ failures is $1 - o(1)$. We condition on that event. For each of those failures, we can use Claim~\ref{c:stage1(2)-ham} to show that Maker can execute (2) using at most $\frac{1}{5b}\log n\cdot \log\log n$ further rounds with probability $1 - O(\frac{1}{n})$. Using union bound, Maker can correct all of the at most $\sqrt{n}\log n$ failures with probability $1 - o(1)$, and takes at most $\frac{1}{5b}\sqrt{n}\log^2 n \cdot \log\log n$ rounds to do so.

In the remaining part of the phase, we have $|V_0| \leq \sqrt{n}$ and thus the number of failures according to (1) is bounded by $\sqrt{n}$. As before, each of those failures can be corrected with probability $1 - O(\frac{1}{n})$ in at most $\frac{1}{5b}\log n\cdot \log\log n$ rounds. Using a union bound, Maker successfully finishes Stage~I after at most $\frac{1}{5b}\sqrt{n}\log n\cdot \log\log n$ further rounds with probability $1 - o(1)$.

In Stage~II Maker tries to connect the two endpoints of $P$. In case of a failure, she again corrects that failure using (2) and succeeds within at most $\frac{1}{5b}\log n\cdot \log\log n$ further rounds with probability $1 - o(1)$ according to Claim~\ref{c:stage1(2)-ham}. In total the last phase a.a.s.~lasts less than $\frac{\eps n}{a} + \sqrt{n}\log^2 n \cdot \log\log n$ rounds as required.
\end{proof}

Using the results from Claim~\ref{c:prob-ai-ham}, Claim~\ref{c:prob-bi-ham}, and Claim~\ref{c:prob-last-ai-ham}, we can lower bound the probability that Maker wins the game using the given strategy by
$$
(1-o(1))\eps^{\sum_j |D_{j-1}|} > \frac{1}{(10b/a)^{21b/a}},
$$
following the argument in the final steps of the proof of Theorem~\ref{thm:maker_degree}~\ref{m_deg_large_b} and using the fact that Breaker can create at most $\frac{2B^{\text{(tot)}}}{0.1 n} = (1+o(1)) \frac{20b}{a}$ vertices of degree at least $\frac{n}{10}$.
\end{proof}

\bigskip


\subsection{Hamiltonicity game with $b \leq a$}

\begin{proof}[Proof of Theorem~\ref{thm:maker_hamilton}~\ref{ham_small_b}]
Let Maker play according to the same strategy
as for Theorem~\ref{thm:maker_hamilton}~\ref{ham_large_b}.

We first consider the phase of the game for which 
$|V_0| \geq \frac14 n$ holds. If Maker needs at most $\frac{3n}{4a} + \frac{1}{5b}\log^2 n\cdot\log\log n$ rounds to reach the point where $|V_0| \leq \frac14 n$, Breaker can claim at most $\frac34 n + \log^2n \cdot \log\log n$ edges in the process. This particularly ensures that Breaker's maximum degree within $V_0$ is much smaller than
$n - \frac{200n}{\log\log n}$.
Then, by the same logic as in the proof of Claim~\ref{c:prob-ai-ham}, Maker will a.a.s.~fail at most $\log n$ times when trying to claim an edge in (1). For each of those failures, she successfully executes (2) with probability at least $1 - O(\frac{1}{n})$ within $\frac{1}{5b}\log n\cdot \log\log n$ rounds according to Claim~\ref{c:stage1(2)-ham}. Therefore, using union bound, Maker a.a.s.~needs at most $\frac{1}{5b}\log^2 n \cdot \log\log n$ rounds for correcting all failures by executions of (2), which additionally ensures the required bound on the number of rounds as mentioned above. Note that at this point, there can be at most one vertex $v \in V_0$ with $d_B(v) \geq \frac12 n$. 

Next, consider the phase until we obtain $|V_0| \leq \frac{200n}{\log\log n}$. Then, if Maker needs no more than $\frac{n}{4a} - \frac{200n}{a\log\log n} + \frac{1}{5b}\log^2 n \cdot \log^2\log n$ rounds to reduce $V_0$ to that size, Breaker can claim at most $n - (1 - o(1))\frac{200n}{\log\log n}$ edges up to this point. In particular, we can still use Claim~\ref{c:stage1(2)-ham} to ensure that each failure in (1) will be corrected by an execution of (2) with probability $1 - O(\frac{1}{n})$. This time the expected number of failures in (1) can be bounded by 
$\frac{(2 - o(1))n}{200n / \log\log n} = O(\log\log n)$, using the same argument as in the proof of Claim~\ref{c:prob-ai-ham}. Thus, by Markov's inequality, the probability to have more than $\log n\cdot \log\log n$ such failures is $o(1)$, and by Claim~\ref{c:stage1(2)-ham} and union bound, Maker can correct all those failures by executions of (2) within at most 
$\frac{1}{5b}\log^2 n \cdot \log^2\log n$ rounds. 
This in particular ensures the mentioned bound on the number of rounds for this phase.
Moreover, since during this phase all but 
$\frac{200n}{\log\log n}$ vertices are removed from $V_0$, 
and almost all the removed vertices
are chosen u.a.r.~according to (1),
the probability that the vertex $v$ with high Breaker degree remains in $V_0$ (if it existed) is bounded by $O(\frac{1}{\log\log n}) = o(1)$. From now on, we condition on the event that $v$ was removed from $V_0$. 

Then, Breaker's maximum degree in $V_0$ is still bounded by $\frac34 n$. Thus, we can follow the proof of Claim~\ref{c:prob-last-ai-ham} to see that Maker a.a.s. finishes the Hamilton cycle in $(1 + o(1))\frac{200n}{a\log\log n}$ further rounds.
\end{proof}

\bigskip


\section{Concluding Remarks}
\label{sec:concluding}

In this paper, we initiated the study of Maker-PhantomBreaker games. While we were able to prove constant lower and upper bounds for the winning probability of Maker for the four studied games, we did not attempt to optimize those bounds. Using our methods, it is possible to improve the lower bounds by choosing some constants and degree thresholds for PhantomBreaker more carefully, but that does not appear to yield anything close to our upper bound. Still, determining the winning probability Maker can achieve for a given bias remains an intriguing question.
\begin{problem}
    For either of the $(a: b)$ biased games -- mindegree-$k$, connectivity, perfect matching or Hamiltonicity,
    with $a,b \in \mathbb{N}$, determine a constant $c(a,b)$ such that Maker has a strategy to win with probability at least $(1 - o(1))c(a,b)$, and PhantomBreaker has a strategy to win with probability at least $(1 - o(1))(1 - c(a,b))$.
\end{problem}
It is not immediately obvious that such a constant must exist, but we believe that it does. Moreover, for mindegree-$1$ game we conjecture that our bound from Theorem~\ref{thm:breaker_degree1} is tight.

\begin{conjecture}
Let $a,b$ be positive integers such that $\frac{b}{2a}>1$ is an integer. Then Maker has a randomized strategy to win the $(a:b)$ mindegree-1 Maker-PhantomBreaker game on $E(K_n)$ with probability at least $(1-o(1))\left( (\frac{b}{2a})!\right)^{-1}$.
\end{conjecture}

There are many variations of this problem that can be studied. It would be interesting to investigate other winning objectives for Maker, such as pancyclicity, non-k-colourability, or non-planarity. One could also study different boards, such as random graphs. Finally, alternative variants of the game could be considered, like Avoider-Enforcer games or Connector-Breaker games.

An obvious follow-up question is what happens if the roles are reversed: Maker plays as the phantom and her moves are hidden, while Breaker's moves are visible. In this case, it quickly becomes clear that Breaker needs a very large bias growing with the number of vertices to have a chance for having a randomized strategy that offers a positive probability of winning in this scenario. We have already made some progress in this direction.

\section*{Acknowledgment}
The research of DC and FH is supported financially by the Federal Ministry of Research, Technology and Space through the DAAD project 57749672.
The research of MM and MS is partly supported by Provincial Secretariat for Higher Education and Scientific Research, Province of Vojvodina (Grant No.~142-451-2686/2021) and partly supported by Ministry of Science, Technological Development and Innovation of Republic of Serbia (Grants 451-03-137/2025-03/ 200125 \& 451-03-136/2025-03/ 200125 and bilateral Serbia-Germany cooperation 001545243 2025 13440 003 000 000 001 03 005). The research of YM was supported by the ANR project P-GASE (ANR-21-CE48-0001-01).

\begin{center}
    \includegraphics[scale=0.18]{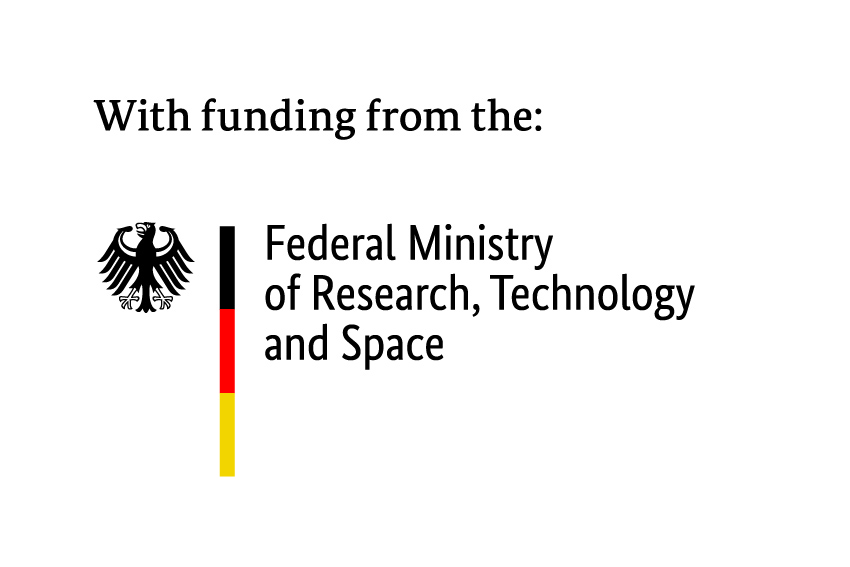}
    \hspace{0.5cm}
    \includegraphics[scale=0.7]{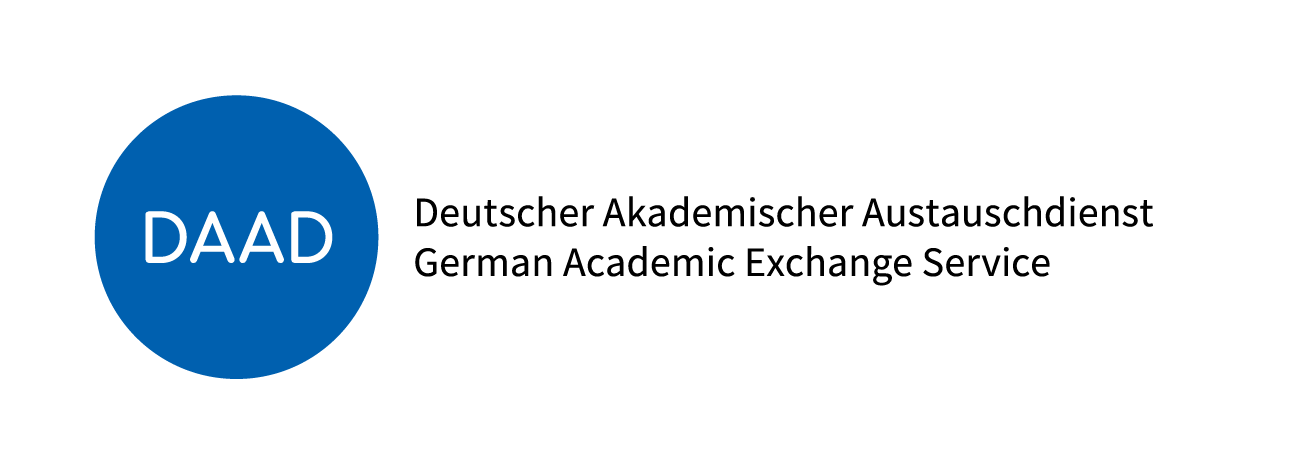}
\end{center}

\bibliographystyle{amsplain}
\bibliography{references}

\end{document}